\documentclass[11pt, reqno]{amsart}
\usepackage{amssymb, graphicx, color}
\usepackage{amsfonts}
\usepackage{amsthm}
\usepackage{enumerate}
\usepackage[mathscr]{eucal}
\usepackage{graphicx}
\usepackage[bookmarksnumbered, bookmarksopen, colorlinks, citecolor=blue, linkcolor=blue]{hyperref}

\allowdisplaybreaks

\newtheorem{thm}{Theorem}[section]
\newtheorem{defi}[thm]{Definition}
\newtheorem{cor}[thm]{Corollary}
\newtheorem{lem}[thm]{Lemma}

\newtheorem{remark}[thm]{Remark}

\theoremstyle{definition}
\theoremstyle{remark}
\numberwithin{equation}{section}

\newcommand{\R}{{\mathbb R}}
\newcommand{\N}{{\mathbb N}}

\newcommand{\C}{{\mathbb C}}

\newcommand{\ten}{{\otimes}}

\newcommand{\bs}{\begin{split}}
\newcommand{\es}{\end{split}}

\newcommand{\be}{\begin{eqnarray*}}
\newcommand{\ee}{\end{eqnarray*}}
\newcommand{\beq}{\begin{align}}
\newcommand{\eeq}{\end{align}}
\newcommand{\s}{\mathscr{S}}

\def\1{\mathbf{1}}

\newcommand{\norm}[2]{\parallel \! #1 \! \parallel_{#2}}

\newcommand{\LG}[1]{\textcolor{red}{#1}}

\begin{document}

\setcounter{page}{1}	
\title[Quantum Boolean functions]
{Geometric influences on quantum Boolean cubes}

%----------Author 1
\author[D. Blecher]{David P. Blecher}
\address{
	Department of Mathematics\\
	University of Houston\\
	Houston, TX 77204-3008\\
	USA}

\email{dblecher@math.uh.edu}
%----------Author 2

\author[L. Gao]{Li Gao}
\address{
	School of Mathematics and Statistics\\
	Wuhan University\\
	Wuhan 430072\\
	China}

\email{gao.li@whu.edu.cn}
%----------Author 3
\author[B. Xu]{Bang Xu}
\address{
	Department of Mathematics\\
	University of Houston\\
	Houston, TX 77204-3008\\
	USA}

\email{bangxu@whu.edu.cn}

%\thanks{This work
	% was partially supported by }


%\date{January 1, 2024.
	%\newline \indent $^{*}$Corresponding author
%}
%----------additions
%\dedicatory{To my boss}
%%%
\begin{abstract}
	In this work, we study three problems related to the $L_1$-influence on quantum Boolean cubes. In the first place, we obtain a dimension free bound for $L_1$-influence, which implies the quantum $L^1$-KKL Theorem result obtained  by Rouze, Wirth and Zhang. %\cite{RWZ}.
Beyond that, we also obtain a high order quantum Talagrand inequality and  quantum $L^1$-KKL theorem.
Lastly, we prove a quantitative relation between the noise stability and $L^1$-influence.  To this end, our technique involves the random restrictions method as well as semigroup theory.
\end{abstract}

\maketitle

\section{Introduction}Boolean functions $f:\{-1,1\}^n\rightarrow\{-1,1\}$,
are central objects in theoretical computer science.
In the past several decades, the Fourier analysis of Boolean functions has been an indispensable tool in the study of Boolean functions, with deep implications in other areas in mathematics and theoretical computer science, such as graph theory, extremal combinatorics, coding theory as well as learning theory. We refer the reader to the excellent books  \cite{CC,O} on Boolean analysis.

Another rapid growing area in the last decades is quantum computation, which studies how to perform computation using quantum mechanic bits (called qubits) instead of classical bits. Correspondingly, Montanaro and Osborne in \cite{MO} initiated the study of quantum Boolean functions, which are defined to be observables on qubits system that takes value (spectrum) of only ${-1,1}$. In recent years, the study of quantum Boolean functions, or more generally Fourier analysis on qubits, has attracted more attention. This includes, the quantum KKL (Kahn-Kalai-Linia) theorem by Rouz\'e, Wirth and Zhang \cite{RWZ}, and the quantum Bohnenblust-Hille inequalities by Volberg and Zhang \cite{VZ}, which have applications in learning theory of quantum systems \cite{HCP}. Given the importance of Boolean analysis in classical computer science, it is foreseeable that quantum Boolean analysis will have more applications in quantum algorithms, quantum learning theory and other directions in quantum computation. Motivated by that,  in this work we further study the quantum version of KKL theorems, proving an improvement and a high order KKL theorem.

For the sake of motivation and history, let us briefly review the concepts in classical Boolean analysis. For a subset $S \subseteq [n] := \{1,2,\ldots,n\}$, the Walsh function $\chi_S: \{-1,1\}^n\rightarrow\R$ is defined as $\chi_S(x) := \prod_{j \in S} x_j$ and $\chi_\emptyset(x) := 1$. The system $\{\chi_S\colon S \subseteq [n]\}$ forms an orthonormal basis of $L^2(\{-1,1\}^n)$, which yields that every $f:\{-1,1\}^n\rightarrow\R$ admits a unique Fourier expansion
$f = \sum_{S \subseteq [n]} \widehat f(S) \, \chi_S$,
where $\widehat f(S) \in \R$.
For each $1\leq j\leq n$, the influence of the $j$-th variable on $f:\{-1,1\}^n\rightarrow\R$ is defined by $$\mathrm{Inf}_j[f]=\sum_{S\subseteq[n]:j\in S}\widehat f(S)^2$$ and the total influence of $f$ is given by
$\mathrm{Inf}[f]=\sum_{j=1}^n\mathrm{Inf}_j[f]$. The notion of influence of Boolean functions was first introduced by Ben-Or and Linial \cite{BOL}, which has become central to Boolean analysis in various contexts since then, including cryptography \cite{LMN93}, hardness of
approximation \cite{DS05,Has01}, and computational lower-bounds \cite{MM}.

The most basic inequality  associated to influence, known as {\em Poincar\'e's inequality,} states that for any $f:\{-1,1\}^n\rightarrow\R$,
one has \begin{align}\mathrm{Var}[f]\leq \mathrm{Inf}[f],\label{eq:PI}\end{align} where $\mathrm{Var}[f]:=\mathbb E[f^2]-\mathbb E[f]^2$ is the variance of $f$. This inequality \eqref{eq:PI} can be understood as a spectrum gap for $\mathrm{Inf}[f]$ as a Dirichlet form. It clearly implies that there exists some index $1\leq j\leq n$ such that $\frac{\mathrm{Var}[f]}{n}\leq\mathrm{Inf}_j[f]$. One may ask that whether $\frac{\mathrm{Var}[f]}{n}\approx\mathrm{Inf}_j[f]$ for all $1\leq j\leq n$, as  Poincar\'e's inequality may be tight in general. Surprisingly, Kahn, Kalai and Linial \cite{KKL88} gave a negative answer to this question. More precisely, the celebrated KKL Theorem states that for any Boolean function $f:\{-1,1\}^n\rightarrow\{-1,1\}$, there exist a coordinate $1\leq j\leq n$ and an absolute constant $C>0$ such that
\begin{equation}\label{KKL}
\mathrm{Inf}_j[f]\geq C \, \frac{\log n}{n} \, \mathrm{Var}[f].
\end{equation}
The inequality (\ref{KKL}) is further strengthened in another famous paper by Talagrand \cite{T}, who proved that there exists an absolute constant $C>0$ such that for any $f:\{-1,1\}^n\rightarrow\{-1,1\}$,
\begin{equation}\label{TAL}
\mathrm{Var}[f]\leq C\sum_{j=1}^n\frac{\mathrm{Inf}_{j}[f]%(1+\mathrm{Inf}_{i,j}^1(A))
}{1+\log(1/\mathrm{Inf}_{j}[f])}.
\end{equation}
 %The reader is referred to \cite{OW1,OW2,CEL12,T2} for the plenty extensions of Talagrand's inequality (\ref{TAL}).
The KKL Theorem and Talagrand's inequality %(\ref{TAL})
are fundamental in the analysis of Boolean functions, and have found many applications such as in learning theory, combinatorics, communication complexity and cryptography. We refer the reader to e.g. \cite{OS07,CKK06,GKK08}  and references therein for more details.

In the quantum setting, the classical hypercube $\{-1,1\}^n$ is replaced by the Hilbert space $H=(\C^2)^{\ten n}$ of $n$-qubits, and the variables $f:\{-1,1\}^n\to \mathbb{R}$ are replaced by the operators on $(\C^2)^{\ten n}$, i.e.,
$M_2(\C)^{\otimes n}\cong M_{2^n}(\C)$ where $M_k(\C)$ denotes the $k$-by-$k$ complex matrix algebra. For $1\leq p\leq\infty$ and $1\leq j\leq n$, the $j$-th $L^p$-influence of $A$ is given by $$\mathrm{Inf}^p_{j}[A]=\|d_j(A)\|_p^p,$$
where $d_j$ is the quantum bit-flip map on $j$-th qubit
$$d_j= \mathbb I^{\otimes(j-1)}\otimes \left(\mathbb I-\frac 1 2\mathrm{tr}\right)\otimes \mathbb I^{\otimes (n-j)}.$$
Here and in what follows, $\mathbb I$ denotes the identity map over $M_2(\C)$ and $\|\cdot\|_p$ denotes the Schatten $p$-norm on $M_2(\C)^{\otimes n}$ with respect to the normalized trace $2^{-n}\mathrm{tr}$ (see section 2 for more details). In the classical case, the $L^1$-influence is often called \emph{the geometric influence} for its relation to isoperimetric inequalities (see e.g. \cite{CEL12}). The total $L^p$-influence of $A$ is given by $$\mathrm{Inf}^p[A]=\sum_{j=1}^n\mathrm{Inf}^p_{j}[A].$$ %It is worth pointing out that when $A$ is a classical Boolean function $\mathrm{Inf}^1_{j}[A]=\mathrm{Inf}^p_{j}[A]$ for all $p\ge 1$ because $d_j(A)\in\{-1,0,1\}$. Nevertheless, $\mathrm{Inf}^1_{j}[A]\neq\mathrm{Inf}^p_{j}[A]$ for  general quantum Boolean functions (see e.g.\ \cite[Section 11]{MO}).
The variance of $A$ is given by $\mathrm{Var}[A]=\|A-2^{-n}\mathrm{tr}(A)\|_2^2.$

Our first result is a  dimension-free bound for $L^1$-influence as the quantum extension of \cite{KKKMS}. %\textcolor{red}{Li: if we do $L_p$-influence what would be the constant $C_p$}
\begin{thm}\label{thm:main1}
	Let $A\in M_2(\C)^{\otimes n}$ with $\|A\|\leq1$. Then there exists an absolute constant $C>0$  such that $$\max_{j}\mathrm{Inf}_j^1[A]\geq2^{-C\frac{\mathrm{Inf}^1[A]}{\mathrm{Var}[A]}}.$$
\end{thm}
This reproduces the quantum KKL theorem for $L_1$-influence:
\begin{align} \max_{j}\mathrm{Inf}_j^1[A]\geq C\frac{\log n}{n} \mathrm{Var}(A) \label{eq:qkkl}\end{align}
by Rouz\'e, Wirth and Zhang from \cite[Theorem 3.6 \& 3.8]{RWZ}, which was proved via the quantum analog of Talagrand's inequality \cite[Theorem 3.2]{RWZ}
\begin{align} \mathrm{Var}[A]\leq C\sum_{j=1}^n\frac{\mathrm{Inf}^1_{j}[A](1+\mathrm{Inf}_{j}^1(A))
}{1+\log^+(1/\mathrm{Inf}^1_{j}[A])}. \label{eq:qt}\end{align}
Here $\log^+$ means the positive part of the logarithm.
 It is not clear that whether the quantum Talagrand's inequality \eqref{eq:qt} implies our Theorem \ref{thm:main1} nor vice versa. We emphasis that our method to $L^1$-influence quantum KKL Theorem \eqref{eq:qkkl} is quite different from \cite{RWZ}. While the proof from \cite{RWZ} mainly uses the semigroups technique, our argument adapts the method of random restriction to the quantum setting. The random restriction method was first introduced by Bourgain \cite{Bou02} as a powerful tool later widely used in Boolean analysis to prove results by induction (e.g. see \cite{KKLMS}). Given the importance of the this method in classical case, we believe our argument is of independent interests and opens up new approaches to other problems in quantum Boolean analysis. For instance, our method also gives an alternative proof for $L^1$-quantum Talagrand's inequality \eqref{eq:qt} (see Theorem \ref{main3}).

%\begin{thm}\label{main2}
%	Let $A\in M_2(\C)^{\otimes n}$ with $\|A\|\leq1$. Then there exists an absolute constant $C>0$ such that \begin{itemize}
%\item[i)] either  $\mathrm{Inf}^1[A]\geq C\log_2 n \, \mathrm{Var}[A]$,
%\item[ii)] or, there exists a coordinate $j\in[n]$ such that $\mathrm{Inf}_j^1[A]\geq\frac{1}{\sqrt{n}}$.
%Hence \end{itemize}
%\begin{align} \max_{j}\mathrm{Inf}_j^1[A]\geq C\frac{\log n}{n} \mathrm{Var}(A)\ . \label{eq:qkkl}\end{align}
	%\end{thm}

Our second result is a quantum Talagrand inequality and quantum KKL theorem for higher-order $L^1$-influences. Let $J=\{j_1,\cdot\cdot\cdot,j_k\}\subset [n]$ be a subset of indices. The $L^p$-influence of $J$ on $A$ is defined as
$$\mathrm{Inf}_{J}^p[A]=\|d_{j_1}\circ d_{j_2}\circ\cdot \circ d_{j_k}(A)\|_p^p. $$
%We refer to Section 2 for the definition of the $k$-th order variance $W_{\geq k}[A]$.
\begin{thm}\label{thm:main2}
	Let $n\geq1$ and $A\in M_2(\C)^{\otimes n}$  with $\|A\|\leq1$. Then for any integer $k$ with $1\leq k\leq n$, we have
\begin{align}W_{\geq k}[A]\leq 24^k\cdot k!\sum_{|J|=k}\frac{\mathrm{Inf}_{J}^1[A]%(1+\mathrm{Inf}_{i,j}^1(A))
}{[\log(1/\mathrm{Inf}_{J}^1[A])]^k},\label{eq:ht}\end{align}
where the summation is over all subset $J$ with $|J|=k$,  and $W_{\geq k}[A]=\sum_{|s|\ge k} |\hat{A}_s|^2$ is the $k$-th order variance (see Section 2 for the definition). As a consequence,
there exists an absolute constant $C>0$ such that
\begin{align}\max_{|J|=k}\mathrm{Inf}_{ J}^1[A]\geq C\Big(\frac{\log n%(1+\mathrm{Inf}_{i,j}^1(A))
}{n}\Big)^k W_{\geq k}[A].\label{eq:ht1}\end{align}
% and $W_{\geq k}[A]=\sum_{\substack{s\in \{0,1,2,3\}^n\\|\mathrm{supp}s|\geq k}}\widehat A_s^2.$
\end{thm}

We remark that the constant $C$ is independent of $k$ as the classical case \cite[Theorem 1.6]{TP} (see also \cite{T1}).

Another important concept we will investigate is noise sensitivity. Let $\delta\in [0,1]$. For $f:\{-1,1\}^n\rightarrow\R$, the noise stability  with parameter $\delta$ of $f$, denote by $S_\delta$, is given by the covariance
$$S_\delta[f]:=\mathrm{Cov}_{x\sim\{-1,1\}^n,y\sim N_\delta(x)}[f(x),f(y)],$$
where $y\sim N_\delta(x)$ denotes that each
bit of $y$ equals to the bit of $x\in\{-1,1\}^n$ with probability $1/2+1/\delta$, and flipped with
probability $1/2-1/\delta$. A sequence $\{f_m:\{-1,1\}^{n_m}\rightarrow\R\}_{m=1}^\infty$ is noise sensitive if for any $\delta>0$, $$\lim_{m\rightarrow\infty}S_\delta[f_m]=0.$$
%The noise operator with parameter $\delta$, denote by $T_\delta$, is given by
%$$T_\delta f(x)=\mathbb E_{y\sim N_\delta(x)}[f(y)]=\sum_{S \subseteq [n]} \delta^{|S|}\widehat f(S) \chi_S(x).$$
Noise sensitivity has been extensively studied in Boolean analysis for its deep connections with measure concentration, threshold phenomena and isoperimetric inequalities (see the book \cite{O} for surveys). %the noise sensitivity and influences of Boolean functions both measure how likely it is to change its value when the input is slightly perturbed. Hence, it makes sense to study the relations between these two concepts.
The connection between noise sensitivity and influences was first established in the work of Benjamini, Kalai and Schramm \cite{BKS}.  Their {\em BKS theorem} states that for any sequence of Boolean functions $\{f_m\}$,
\begin{align}\sum_{j=1}^{n_m}\mathrm{Inf}_j[f_m]^2\underset{m\to\infty}{\longrightarrow }0\ \ \Longrightarrow\ \   \{f_m\} \text{ is noise sensitive. } \label{eq:BKS}\end{align} %\ \ \mbox{implies}\ \ \mathrm{Var}[T_\delta \, f_m]\underset{m\to\infty}{\longrightarrow }0\ \ \forall\delta\in(0,1).$$

In the quantum setting, the noise stability of a operator $A\in M_2(\mathbb{C})^{n}$ can be defined by the normalize trace inner product
\[ S_\delta(A):=\frac{1}{2^n}\mathrm{tr}(A^*T_\delta(A)), \]
where  $T_\delta=(\delta \mathbb{I}+(1-\delta) \frac{1}{2}\mathrm{tr})^{\ten n}$ is the tensor product depolarizing map. Our third result is a quantum generalizations of the quantitative BKS theorem obtained by Keller and Kindler \cite{KK}. %established a quantitative version between Noise operator and influences by using the tool of Fourier-Walsh expansions for Boolean functions. More precisely, they

\begin{thm}\label{thm:main3}
	Let $\delta\in(0,1)$, $n\geq1$ and $A\in M_2(\C)^{\otimes n}$ with $\|A\|\leq1$. Then  there exist two absolute constants $0<C_0<1$ and $C_1>0$ such that $$S_\delta [A]\leq C_1\Big(\sum_{j\in[n]}\mathrm{Inf}_{j}^1[A]^2\Big)^{C_0(1-\delta)}.$$
\end{thm}

The rest of this paper is organized as follows. In Section 2, we recall the basics of Fourier analysis on quantum Boolean cubes. This section also reviews some necessary properties of quantum depolarizing semigroups and the noise operators including  hypercontractivity and logarithmic Sobolev inequality. Section 3 proves the   dimension-free bound of $L^1$-influence in Theorem \ref{thm:main1}, and also gives an alternative proof of the quantum $L^1$-Talagrand's
inequality (Theorem \ref{main3}). Section 4 is devoted to the quantum Talagrand inequality for higher order $L^1$-influence (Theorem \ref{thm:main2}). We prove the quantitative BKS theorem (Theorem \ref{thm:main3}) in Section 5. Section 6 includes an argument for the quantum FKN theorem fixing an error from \cite{MO}. We end the paper with a discussion on open problems and further directions.

\medskip

{\bf Note}: After the first version of the paper was completed, we were kindly informed by Dejian Zhou that Theorem \ref{thm:main1} has been independently discovered in a work in preparation by Yong Jiao, Sijie Luo and Dejian Zhou.

{\bf Acknowledgement}: This project is partially supported by NSF grant DMS-2154903. D.B. was partially supported by a Simons Foundation Travel/Collaboration grant. L.G. was also supported by National Natural Science Foundation of China.

\section{Preliminaries on Quantum Boolean Analysis}
In this section, we set up the notations and definitions that we will work with. The readers are referred  to \cite{MO,RWZ} for more information on quantum Boolean analysis. In the sequel, we write $X\lesssim Y$ (resp. $X\gtrsim Y$) if $X\leq C\, Y$ (resp. $X\geq C \, Y$) for some absolute constant $C$, and $X\approx Y$ if $X\lesssim Y$ and $X\gtrsim Y$.
\subsection{Fourier analysis on quantum Boolean hypercube}\label{qBoolean}
The (classical) Boolean analysis studies the functions on the Boolean hypercube $\{-1,1\}^n$, i.e. functions on a $n$-bits string. The quantum analog considers the observables on $n$-qubits system, which forms the matrix algebra $M_2(\C)^{\otimes n}$ on the complex Hilbert space $H=(\C^2)^{\ten n}$. We denote by $\mbox{tr}$ the standard matrix trace on $M_2(\C)^{\otimes n}\cong M_{2^n}(\C)$ and $\tau=\frac 1{2^n}\,\mathrm{tr}(\cdot)$ by the normalized trace. We will use the normalized Schatten-$p$ norm of $A\in M_2(\C)^{\otimes n}$ defined as %($|A|:=(A^\ast A)^{1/2}$)
\begin{align*}
	\|A\|_p=\tau(|A|^p)^{\frac{1}{p}},
\end{align*}
where $|A|:=(A^\ast A)^{1/2}$. For $p=\infty$, $\|\cdot\|_{\infty}\equiv \|\cdot\|$ is the usual operator norm.

Recall that Pauli matrices are self-adjoint unitaries
\begin{equation*}
	\sigma_0=\begin{pmatrix}1&0\\0&1\end{pmatrix},\quad \sigma_1=\begin{pmatrix}0&1\\1&0\end{pmatrix},\quad \sigma_2=\begin{pmatrix}0&-i\\i&0\end{pmatrix},\quad \sigma_3=\begin{pmatrix}1&0\\0&-1\end{pmatrix},
\end{equation*}
which form an orthonormal basis of $M_2(\C)$, with respect to the normalized trace inner product. For $n$-qubits system,
we have the tensor product of Pauli matrices: for a multi-index $s=(s_1,\dots, s_n)\in\{0,1,2,3\}^n$,
\begin{equation*}
	\sigma_s:=\sigma_{s_1}\otimes\dots\otimes \sigma_{s_n}\, ,
\end{equation*}
form an orthonormal basis of $M_2(\C)^{\otimes n}$. Accordingly, every $A\in M_2(\C)^{\otimes n}$ has an unique Fourier expansion
\begin{equation}\label{eq:quantum_fourier_expansion}
	A=\sum_{s\in\{0,1,2,3\}^n} \widehat A_s \,\sigma_s,
\end{equation}
where %\begin{equation*}
$\widehat A_s=%\frac 1{2^n}\,\mathrm{tr}(\sigma_s A)=
\tau(\sigma_s A)\in \C\,$
%\end{equation*}
are the Pauli coefficients. %\LG{(Shall we call it Pauli coefficient)?}

Denote $[n]=\{1,\dots,n\}$. Given an multi-index $s=(s_1,\dots, s_n)\in\{0,1,2,3\}^n$, we denote the {\em support} of $s$ as
\[ \mbox{supp} \, s :=\{j\in [n] \ |\ s_j\neq0 \}\]
and $|\mbox{supp} \, s|$ as its cardinality. We say that $A\in M_2(\C)^{\otimes n}$ is of {\em degree} $d$ if $\widehat A_s=0$ whenever $|\mbox{supp}  \,  s|>d$, and it is called of $d$-{\em homogeneous} if $\widehat A_s=0$ for $|\mbox{supp}  \,  s|\neq d$. For each integer $d\geq1$, the {\em  Fourier weight} of $A\in M_2(\C)^{\otimes n}$ at degree $d$ is $$W_{=d}[A]=\sum_{\substack{s\in \{0,1,2,3\}^n\\|\mbox{supp}  \, s|=d}}| \widehat{A}_s|^2.$$
We will also use the notation
$$ W_{\geq d}[A]=\sum_{\substack{s\in \{0,1,2,3\}^n\\|\mbox{supp}  \, s|\geq d}}
|\widehat{A}_s|^2 ,\  \; \; \mbox{and}\ \; \; W_{\approx d}[A]=\sum_{d\leq j<2d}W_{=j}[A].$$
%\widehat A_s^2.$$

\subsection{$L^p$-influence} %For $A\in M_2(\C)^{\otimes n}$.
Following \cite[Definition 10.1]{MO} (see also \cite{RWZ}),
the $j$-th {\em derivative operator} $d_j$ with $j\in [n]$ is defined as %(see \cite[Definition 10.1]{MO} and \cite{RWZ})
\begin{equation}\label{eq:quantum_fourier_expansion1}
	d_j(A):= \mathbb I^{\otimes(j-1)}\otimes (\mathbb I-E)\otimes \mathbb I^{\otimes (n-j)}(A)
	=\sum_{\substack{s\in \{0,1,2,3\}^n\\s_j\neq 0}}  \widehat A_s \; \sigma_s\,.
\end{equation}
where $\mathbb I$ denotes the identity map and $E(X)=\frac{1}{2}\mbox{tr}(X) {\bf 1}$ is the trivial conditional expectation on $M_2(\C)$. Equivalently, one can write
$d_j=\mathbb I-E_j$, where
\[E_j= \mathbb I^{\otimes(j-1)}\otimes E\otimes \mathbb I^{\otimes (n-j)}\] is the conditional expectation onto the subalgebra $M_2(\C)^{\otimes(j-1)}\otimes \mathbb{C}I \otimes M_2(\C)^{\otimes(n-j)}$ trivial on the $j$th-qubit. Hence $d_j\circ d_j=d_j$ is an idempotent and an orthogonal projection with respect to the $2$-norm.

For $p\ge 1$, define the \textit{$j$th $L^p$-influence}  and  the \textit{total $L^p$-influence} as follows
\[\mbox{Inf}^p_j[A]:=\|d_j (A)\|_p^p\, \ ,\ \,  \mbox{Inf}^p [A]:=\sum_{j=1}^{n}\mbox{Inf}_j^p(A)\ .\]
%For our purposes we will be mainly concerned with $L^1$ and $L^2$-influences.
The $L^1$-influence is also called {\em the geometric influence}. We will write $\mathrm{Inf}_j[A]=\mathrm{Inf}^2_j[A]$ and $\mathrm{Inf}[A]=\mathrm{Inf}^2[A]$, as is customary.
Observe that
\begin{align}\label{rwz25}
	\mathrm{Inf}[A] = \sum_{j \in [n]} \, \sum_{s: s_j \neq 0} \, | \widehat{A}_s |^2
	= \sum_s \,  \sum_{j \in {\rm supp} \, s} \, | \widehat{A}_s |^2 =
	\sum_s \, |{\rm supp} \, s | \, | \widehat{A}_s |^2.
\end{align}
\begin{lem} \label{was21}
	Let $A\in M_2(\C)^{\otimes n}$ with $\|A\|\leq1$. Then for any $1\leq p<2$ and each $j\in[n]$,
	\begin{equation}\label{inf}
		2^{p-2}\mathrm{Inf}_j[A]\leq \mathrm{Inf}_j^p[A]\leq \big(\mathrm{Inf}_j [A]\big)^{p/2}\le 1.
	\end{equation}
	As a consequence, $\mathrm{Inf}[A]\leq2^{2-p}\mathrm{Inf}^p[A]$.
	%\LG{As a consequence
		%\[ \frac{1}{2}\mathrm{Inf}[A]\leq \mathrm{Inf}^1[A]\leq \frac{\sqrt{n\mathrm{Inf}[A]}}{2}.\]}
\end{lem}
\begin{proof}
	Since $d_j= I-E_j$, $\|d_j\|_{\infty\rightarrow\infty}\leq2$.
	It  then follows from H\"older's inequality that
	$$\|d_j(A)\|_2\leq\|d_j(A)\|_p^{p/2}\|d_j(A)\|_\infty^{1-p/2}\leq2^{1-p/2}\|d_j(A)\|_p^{p/2},$$
	which gives $\|d_j(A)\|^2_2\leq2^{2-p}\|d_j(A)\|_p^p$. By H\"older's inequality once more with $q$ being the conjugate of $p/2$,
	$$\|d_j(A)\|_p^p\leq\big(\tau(1)\big)^{1/q}\big(\tau(|d_j(A)|^2)\big)^{p/2}=\big(\mathrm{Inf}_j [A]\big)^{p/2}\le \|A\|_2^{p}\leq\|A\|^{p}\leq1,$$
	as required.
\end{proof}

%dBy definition, $d_j^2=d_j$ for every $j\in[n]$. It then

\subsection{Depolarizing Semigroup and Noise operators}
Since $d_j=\mathbb{I}-E_j$ is an idempotent, it follows that
\begin{equation}\label{eq:l2_influence_dirichlet}
	\mathrm{Inf}[A]=\sum_{j=1}^n \tau((d_j A)^\ast d_j(A))=\tau(A^\ast \mathcal L(A)),
\end{equation}
which is the Dirichlet form with the generator $\mathcal L:=\sum_{j=1}^n d_j$. Note that $d_j\circ d_i=d_i\circ d_j$. The $\mathcal{L}$ generates the tensor product {\em quantum depolarizing semigroup}  $(P_t)_{t\ge 0}: M_2(\C)^{\otimes n}\to M_2(\C)^{\otimes n}$% for the individual qubits:
\begin{equation*}%\label{eq:defn of semigroup}
	P_t=e^{-t\mathcal L}=\left(e^{-t}\,\mathbb I+(1-e^{-t})E\right)^{\otimes n}\underset{t\to\infty}{\longrightarrow }E^{\otimes n}:=E_{[n]}\,,
\end{equation*} where $E_{[n]}(\rho)=\tau(\rho){\bf 1}$ is the completely depolarizing map on $M_2(\C)^{\otimes n}$.
We have the following expression for $P_t$ via Fourier series:
\begin{equation}\label{eq:defn of semigroup1}
	P_t \,  A=\sum_{s\in\{0,1,2,3\}^n}e^{-t|\mbox{supp}  \, s|} \widehat A_s \,\sigma_s.
\end{equation}

We collect some further properties for $(P_t)_{t\geq0}$ as follows. Recall that the variance of a operator $A$ is
\[\text{Var}[A]=\|A-E_{[n]}(A)\|^2_2\ ,\]
and the entropy of a nonnegative operator $\rho$ is defined as
$$\mbox{Ent}[\rho]=\tau(\rho\log\rho)-\tau(\rho)\log\tau(\rho),$$
with the convention $0\log0=0$.
\begin{lem}\label{basic}Let $(P_t)_{t\ge 0}: M_2(\C)^{\otimes n}\to M_2(\C)^{\otimes n}$ be the tensor depolarizing semigroup above. We have
	\begin{itemize}
		\item[i)] \emph{(Poincar\'e inequality)} For any $A\in M_2(\C)^{\otimes n}$ and $t\geq 0$,
		\begin{align}\|P_t(A)-E_{[n]}(A)\|_2\leq e^{-t}\|A-E_{[n]}(A)\|_2. \label{eq:poincare}\end{align}
		Equivalently, one has
		\begin{align}\mathrm{Var}[A]\le \mathrm{Inf}[A]\ .\label{eq:poincare2}\end{align}
		\item[ii)] \emph{(Hypercontractivity)} For any $A\in M_2(\C)^{\otimes n}$ and $p(t)= 1+e^{-2t}$ with $t\ge 0$,
		\begin{align}
			\|P_t(A)\|_2\leq \|A\|_{p(t)}\,.
			\label{eq:hyper}\end{align}
		Equivalently, one has the log-Sobolev inequality that for any $A\in M_2(\C)^{\otimes n}$
		\begin{align}\mathrm{Ent}[|A|^2]\leq 2\mathrm{Inf}[A]\ .\label{eq:lsi}\end{align}
		As  a consequence, for any   $A$ of  degree at most $d\geq 2$ and for $q\geq2$,
		\begin{align}%\label{e21222}
			\|A\|_q\leq(q-1)^{d/2}\|A\|_2.
			\label{eq:degree}\end{align}
		\item[iii)] \emph{(Intertwining relation)} For any $j\in\{1,\dots,n\}$ and $t\geq 0$,
		\begin{align}
			d_j P_t=P_t d_j.
			\label{eq:degree}\end{align}
	\end{itemize}
\end{lem}
The Poincar\'e inequality i) and hypercontractivity ii) are obtained in \cite[Section 8 \& 9]{MO} whenever $A$ is self-adjoint. The extension to general operators is a standard Cauchy-Schwarz type argument (see \cite[Theorem 12]{DJKR}) using $P_t$ are completely positive maps. The equivalence between the hypercontractivity and the
log-Sobolev inequality in the noncommutative setting are also well-known  (see \cite[Theorem 3.8]{OZ}). The log-Sobolev inequality for tensor depolarizing semigroup $P_t$ was also proved directly in \cite{TK}. %\cite[Corollary 8.9]{MO}.

Following \cite[Definition 8.1]{MO}, we will also use the notion of the noise operator $T_\delta$, which defined as
\begin{equation}\label{noise}
	\ T_\delta(A)=\sum_{s\in\{0,1,2,3\}^n} \delta^{|\mbox{supp}  \, s|} \, \widehat A_s \,\sigma_s,\  \quad \delta\in(0,1)
\end{equation}
with $T_1=\mathbb{I}$ and $T_0(A)=\widehat A_0{\bf 1}$.
Observe that $T_{e^{-t}}=P_t$. Hence we have contractivity that for any $1\leq p\leq \infty$ and $A\in  M_2(\C)^{\otimes n}$
\begin{align}\label{eq:conctract}\norm{T_\delta(A)}{p}\le \norm{A}{p}\end{align}
and the hypercontractivity: for any $1\leq p\leq2\leq q\leq\infty$,
\begin{align}\label{noise12}
	\|T_\delta A\|_q\leq\|A\|_p\ \ \ \ \text{ if } \ \delta\leq\sqrt{\frac{p-1}{q-1}}.
\end{align}	
The noise stability of $A$ is defined as
\begin{align}\label{noise1222}
	S_\delta[A]=\sum_{s\neq0} \, \delta^{|\mathrm{supp} \, s|} \, |\widehat A_{s}|^2.
\end{align}

We have the following easy consequence of the log-Sobolev inequality \eqref{eq:lsi}.
\begin{lem}\label{log-sobo}
	%There exists an absolute constant $K>0$ such that
	For any $A\in M_2(\C)^{\otimes n}$ with $\|A\|\leq1$ and for any $1\leq p<2$  %\textcolor{red}{(should $A$ be invertiable?)}
	$$\mathrm{Inf}[A]\geq-\frac{1}{2}\|A\|_2^2\log\|A\|_2^2-\frac{1}{2-p}\|A\|_p^{p/2}\|A\|_2.$$
\end{lem}
\begin{proof}
	By \eqref{eq:lsi} and the definition of entropy, it is enough to show
	\begin{align}\label{log-sob}
		\tau(|A|^2\log |A|^{-2})\leq\frac{2}{2-p}\|A\|_p^{p/2}\|A\|_2.
	\end{align}	
	Indeed, it follows from the Cauchy-Schwarz inequality that
	$$\tau(|A|^2\log |A|^{-2})\leq\tau(|A|^2\log^2(|A|^{-2}))^{1/2}\|A\|_2\leq\frac{4}{(2-p)e}\|A\|_p^{p/2}\|A\|_2,$$
	where in the last inequality we used the assumption $\|A\|\leq1$, and
	the basic inequality $t^2\log^2(1/t^2)\leq\frac{16}{e^2}\frac{1}{(2-p)^2}|t|^p$ for $t\in[-1,1]$.
	This finishes the proof.	
\end{proof}

%Note that $E_J$ is the conditional expectation for a fixed $J$ not the expectation $\mathbb{E}_J$ over a random subset $J$.

\subsection{Conditional Expectations and Partial Fourier Series}
Let $J\subset [n]$ be a subset and $J^c$ be its complement. Denote ${\bf 1}_J$ as the identity operator in $M_2(\C)^{J}$. We introduce the normalized partial trace $\tau_J$ and the conditional expectation $E_J$ over a subset $J$ as
\begin{align}&\tau_J:M_2(\C)^{\otimes n}\to M_2(\C)^{J^c}\ ,\  \tau_J(X)=\tau_J\ten \mathbb{I}_{J^c}(X)\nonumber
	\\ & E_J:M_2(\C)^{\otimes n}\to M_2(\C)^{\otimes n}\ , \  E_J(X)=\tau_J(X)\ten {\bf 1}_J\label{eq:partial},
\end{align}
where $\tau_J$ is the normalized matrix trace on $M_2(\C)^{J}$ and $\mathbb{I}_{J^c}$ is the identity map on its complement.
These two maps are essentially the same map by identifying $M_2(\C)^{J^c}\cong M_2(\C)^{J^c}\ten (\mathbb{C}{\bf 1}_J)\subset M_2(\C)^{\otimes n}$. These conditional expectations mutually commutes and satisfies
\[ E_I\circ E_J =E_J\circ E_I=E_{I\cup J}\ , E_J=\prod_{j\in J}E_j\ . \]

For every $A\in M_2(\C)^{\otimes n}$, it admits an unique partial Fourier expansion over $J$:
\[A=\sum_{t\in \{0,1,2,3\}^J} \widehat{A}_{t,J}\sigma_{t}\]
where $\sigma_{t}\in M_2(\C)^{J}$ and its coefficient
\[ \widehat{A}_{t,J}=\tau_J(A \sigma_{t})\]
Here the multiplication $\widehat{A}_{t,J}\sigma_{t}$ is a shorthand for the tensor product $\widehat{A}_{t,J}\otimes \sigma_{t}\in  M_2(\C)^{J^c}\otimes M_2(\C)^{J}\cong M_2(\C)^{\otimes n}$. In the sequel, $t\in \{0,1,2,3\}^{J}$ means that $t\in \{0,1,2,3\}^{n}$ and $\mathrm{supp}\ t\subseteq J$. For $s\in \{0,1,2,3\}^{J^c}$ and $t\in \{0,1,2,3\}^{J}$, we write $s\cup t$ as the joint element in $\{0,1,2,3\}^{n}$. Then $\widehat{A}_{t,J}$ admits a further expansion
\[ \widehat{A}_{t,J}=\sum_{s\in \{0,1,2,3\}^{J^c}} \widehat{A}_{s\cup t}\sigma_{s}\ , \ \|\widehat{A}_{t,J}\|_{2}^2=\sum_{s\in \{0,1,2,3\}^{J^c}} |\widehat{A}_{s\cup t}|^2\ .\]

\subsection{Random Subsets}
The idea of random subsets play an important role in many works of Boolean analysis. We refer to e.g. \cite{Bou02,KS1,DFKO05,KK,KKKMS,KKLMS} for the applications of random subset.

\begin{defi}[Random subset]
	Let $\delta\in [0,1]$.
	We say $J \subset [n]$ is a $\delta$-random subset of $[n]$ if for any subset $I\subset [n]$, the probability $\mathrm{Pr}[J=I]=(1-\delta)^{n-|I|}\delta^{|I|}$.
\end{defi} By the identification \[\{0,1\}^n\cong 2^{[n]}\  ,\  r=(r_1,\cdots,r_n)\mapsto  \mbox{supp}\ r , \]
$J$ is basically the joint $n$-copies of the i.i.d.\  Bernoulli random variables with $\mathrm{Pr}[r_i = 1]=\delta$.

Given $v=(v_1,v_2,\cdots, v_n)\in \{0,1,2,3\}^{n}$, we are interested in the event that a random subset $J$ intersect with $\mbox{supp}\, v$ at exactly one site, whose probability is
\[\mathrm{Pr}[|J \cap {\rm supp}\, v| =1]= |\mbox{supp}\ v|\delta (1-\delta)^{|\mathrm{supp}\, v|-1}\]
By Fubini's theorem, we have the following lemma.
%\LG{The lemma below is changed. Please check.}
\begin{lem}\label{first level}
	Let $J$ be  a $\delta$-random subset of $[n]$ with $\delta\in[0,1]$. Then for $A\in M_2(\C)^{\otimes n}$,
	\begin{align*}
		&\sum_{v\in\{0,1,2,3\}^n} \Big(|\widehat A_v|^2 \; \mathrm{Pr}[|\mathrm{supp}\, v \cap J|=1]\Big)
		=\mathbb E_J\Big[\sum_{\substack{v\in \{0,1,2,3\}^{n}\\
				|\mathrm{supp}\, v \cap J|=1}} |\widehat A_{v}|^2\Big].
	\end{align*}
	Moreover, $\mathrm{Pr}[J \cap {\rm supp}\, v| =1]= |\mathrm{supp}\, v|\delta (1-\delta)^{|\mathrm{supp}\, v|-1}$.
\end{lem}

\section{Proof of Theorem \ref{thm:main1}}
This section is devoted to the proof of Theorem \ref{thm:main1}, which reproduces the quantum KKL theorem for $L^1$-influence. Different with \cite{RWZ} using hypercontractivity and semigroup techniques, we use the random restriction method and log-Sobolev inequality (equivalent form of hypercontractivity at infinitesimal level) following the idea from \cite{KKKMS}. %\LG{please give the reference on which classical work we are following.}

\subsection{A key lemma}

We start with the following lemma, which is crucial in proving Theorem \ref{thm:main1}.
\begin{lem}\label{bmo}
	Let $A\in M_2(\C)^{\otimes n}$ with $\|A\|\leq1$ and $J\subseteq[n]$. %and $k\in\{1,2,3\}$. % such that (\ref{e}) holds.
	Then %there exists an absolute constant $K>0$ such that
	$$\mathrm{Inf}[A]\geq\frac{1}{2} \log\Big(\frac{1}{\max_{j\in J}2\mathrm{Inf}_j^1[A]}\Big)\cdot W_{J}(A)-\sqrt{3}\sqrt{\mathrm{Inf}^1 [A]}\cdot\sqrt{W_{J}(A)},$$
	where
	$$W_{J}(A)=\sum_{\substack{v\in \{0,1,2,3\}^{n}\\
			|{\mathrm{supp}\, v \cap J|=1}}} |\widehat{A}_v|^2.$$
\end{lem}
%To illustrate the meaning of the term $W_J$, we set some necessary notations.
Here $W_J$ is the $L_2$-weight of the Fourier terms of $A$ such that $|\mathrm{supp}\, v \cap J|=1$. It can be decomposed as
$W_J=\sum_{k=1}^3W_{J,k}$ and $W_{J,k}$ is given by
$$W_{J,k}(A)=\sum_{j\in J}\sum\limits_{s\in \{0,1,2,3\}^{J^c}}\sum\limits_{\substack{t\in \{0,1,2,3\}^{J}\\  \mathrm{supp} \, t=\{j\},t_j=k}} |\widehat A_{s\cup t}|^2.$$
For each $k\in\{1,2,3\}$, define $g_{j,k}$ as the element
\begin{align}\label{e12}
	g_{j,k}:=\sum\limits_{s\in \{0,1,2,3\}^{J^c}}\sum\limits_{\substack{t\in \{0,1,2,3\}^{J}\\ \mathrm{supp} \, t=\{j\},t_j=k}}%\sum\limits_{s\in \{0,1,2,3\}^{J^c}}
	\widehat A_{s\cup t} \sigma_{s}.
\end{align}
Note that the second summation above contains just one term $t=k(j)$, where we denote by \[k(j)=(0,\dots, 0, \underset{j\text{th}}{k}, 0\cdots, 0) \in\{0,1,2,3\}^{J}\] the index  being $k$ at site $j$ and vanishes elsewhere. Then $g_{j,k}$ is nothing but the partial Fourier coefficient
\[ g_{j,k}=\sum\limits_{s\in \{0,1,2,3\}^{J^c}}
\widehat A_{s\cup k(j)} \sigma_{s}=\tau_J(A\sigma_{k(j)}),\] %\widehat{A}_{k(j),J},
where $\tau_J$ was defined in (\ref{eq:partial}) %denotes the normalized partial trace restricted to the subset $J$
and
$\sigma_{k(j)}=
\sigma_0\otimes\cdots \otimes\sigma_k\otimes\cdots \otimes\sigma_0\in M_{2}(\mathbb C)^{\otimes n}$. Using these notations, we have
\[ W_{J,k}(A)=\sum_{j\in J} \norm{g_{j,k}}{2}^2 \ , \ W_{J}(A)=\sum_{k=1}^3\sum_{j\in J} \norm{g_{j,k}}{2}^2.\]

%Define $g_{j,k}$ as \LG{here the second summation is trivial. The definition $g_{j,k}$ also depending on $J$.} %a restricted operator of $A$. More precisely, %the $j$-th coefficient of the first level of $A_{s,J^c}$. That is,
%$g_{j,s}=\widehat A_{\{j\}\cup s} \sigma_{j}%\widehat A_{j,J}=\sum_{\substack{t\in \{0,1,2,3\}^{n}\\\textbf{t}\cap J=j}}\widehat A_{t} \sigma_{\textbf{t}\setminus j}.
%$. Then by definition, $W_{=1}[ A_{s,J^c}]=\sum_{j\in J}\|g_{j,s}\|_2^2$. We define $g_j$ as
\iffalse
\begin{align}\label{e12}
	g_{j,k}=\sum\limits_{s\in \{0,1,2,3\}^{J^c}}\sum\limits_{\substack{t\in \{0,1,2,3\}^{J}\\ \mathrm{supp} \, t=\{j\},t_j=k}}%\sum\limits_{s\in \{0,1,2,3\}^{J^c}}
	\widehat A_{s\cup t} \sigma_{s}.
\end{align}
Denote \[k(j)=(0,\dots, 0, \underset{j\text{th}}{k}, 0\cdots, 0) \in\{0,1,2,3\}^{J}\] as the index equals $k$ at site $j$ and vanishes else where. The second summation above contains just one term $t=k(j)$. The $g_{j,k}$ is nothing but the partial Fourier coefficient
\[ g_{j,k}=\sum\limits_{s\in \{0,1,2,3\}^{J^c}}
\widehat A_{s\cup k(j)} \sigma_{s}=\tau_J(A\sigma_{k(j)})= \widehat{A}_{k(j),J}\]
where$\tau_J$ denotes the normalized partial trace restricted to the subset $J$ and
$\sigma_{k(j)}=
\sigma_0\otimes\cdots \otimes\sigma_k\otimes\cdots \otimes\sigma_0\in M_{2}(\mathbb C)^J$. Then
\[ \norm{g_{j,k}}{2}^2=\sum\limits_{s\in \{0,1,2,3\}^{J^c}}\widehat |A_{s\cup  k(j)}|^2\]

\fi

%\LG{we actually obtain $\max_{j\in J}\mathrm{Inf}_j[A]$ here}

The proof of Lemma \ref{bmo} depends on the following three lemmas.

%where $\widehat A_{\{j\}}$ equals to $\widehat A_{t}$ with $\mathrm{supp}t=j$.
%Then by definition, $W_{=1}[ A_{s,J^c}]=\sum_{j\in J}\|g_{j,s}\|_2^2$.
%Then (\ref{e}) now is restated as
%\begin{align}\label{e1}
%\sum_{\substack{s\in \{0,1,2,3\}^{|J^c|}\\\mathrm{supp}s\subseteq J^c}}	
%	\sum_{j\in J}\|g_{j}\|_2^2\gtrsim W_{\approx d}[A].
%\end{align}
\begin{lem}\label{restrict}  Let $A\in M_2(\C)^{\otimes n}$ and $J\subseteq[n]$.
	For each $j\in J$ and $k\in\{1,2,3\}$,	
	%\begin{align}\label{ress}
	$$	g_{j,k}=\tau_J(A \, \sigma_{k(j)})=\tau_J(d_j(A)   \, \sigma_{k(j)}).$$
	%\end{align}
	%\sigma_{\{0,\cdot\cdot\cdot0,s_j,0,\cdot\cdot,0\}}$
\end{lem}
\begin{proof}
	Let $j\in J$ and $k\in\{1,2,3\}$.  By the identification
	$E_J(X)=\tau_J(X)\ten {\bf 1}_J$,
	it is equivalent to show
	\[ E_J(A\sigma_{k(j)})=E_J(d_j(A)\sigma_{k(j)}).\]
	Since $d_j=\mathbb{I}-E_j$,
	\begin{align*}
		E_J(d_j(A)\sigma_{k(j)})=E_J((\mathbb{I}-E_j)(A)\sigma_{k(j)})=E_J(A\sigma_{k(j)})-E_J(E_j(A)\sigma_{k(j)}).
	\end{align*}	
	Hence, it remains to show that $E_J(E_j(A)\sigma_{k(j)})=0$. Indeed,
	\[E_J(E_j(A)\sigma_{k(j)})=E_{J/{j}} \Big(E_j(E_j(A)\sigma_{k(j)})\Big)= E_{J/{j}} \Big(E_j(A) E_j(\sigma_{k(j)})\Big)=0,\]
	where we have used the module property of condition expectation $E(axb)=aE(x)b$ for $a,b$ in the range of $E$ and \[ E_j(\sigma_{k(j)})=\frac{1}{2}\mathrm{tr}_{\{j\}}(\sigma_k(j)){\bf 1}=0\]
	for $k\neq 0$. This completes the proof.
	\iffalse
	\begin{align*}
		\frac{1}{2^{|J|}}\mathrm{tr}_J(A  \,  \sigma_{k(j)})
		= &\  \frac{1}{2^{|J|}}\mathrm{tr}_J\Big(\sum\limits_{\substack{s\in \{0,1,2,3\}^{J^c}\\|\mathrm{supp} \, t|=1,t_j=k}}	\widehat A_{s\cup t}  \, \sigma_{s\cup t}  \,  \sigma_{k(j)}\Big)\\
		&\ + \frac{1}{2^{|J|}}\mathrm{tr}_J\Big(\sum\limits_{\substack{s\in \{0,1,2,3\}^{J^c}\\|\mathrm{supp} \, t|=1,t_j\in\{1,2,3\}\backslash\{k\}}}	\widehat A_{s\cup t}  \, \sigma_{s\cup t}  \, \sigma_{k(j)}\Big)\\
		&\ +\frac{1}{2^{|J|}}\mathrm{tr}_J\Big(\sum\limits_{\substack{s\in \{0,1,2,3\}^{n}\\s_j=0}}\widehat A_s  \, \sigma_{s}  \,  \sigma_{k(j)}\Big) .
	\end{align*}	
	By using the basic properties of Pauli matrices, we can see that the last two terms on the right hand side above equal to 0 and
	$$\frac{1}{2^{|J|}}\mathrm{tr}_J\Big(\sum\limits_{\substack{s\in \{0,1,2,3\}^{J^c}\\|\mathrm{supp} \, t|=1,t_j=k}}	\widehat A_{s\cup t}\sigma_{s\cup t}\sigma_{k(j)}\Big)=\sum\limits_{\substack{s\in \{0,1,2,3\}^{J^c}\\|\mathrm{supp} \, t|=1,t_j=k}}
	\widehat A_{s\cup t} \sigma_{s}.$$
	This proves the first equation. For  the second equation, it follows from (\ref{eq:quantum_fourier_expansion1}) that
	$$A \sigma_{k(j)}=d_j(A) \sigma_{k(j)}+\sum_{\substack{s\in \{0,1,2,3\}^n\\s_j= 0}}\widehat A_s \sigma_s \sigma_{k(j)}.$$
	Hence, the statement follows by noting
	$$\frac{1}{2^{|J|}}\mathrm{tr}_J\Big(\sum_{\substack{s\in \{0,1,2,3\}^n\\s_j= 0}}\widehat A_s \sigma_s \sigma_{k(j)}\Big)=0.$$
	The latter  is because in the $j$th `position' we have a Pauli matrix with trace 0.
	This completes the proof.
	\fi	
\end{proof}
%For each $j\in J$ and $k\in\{1,2,3\}$, we set %$p_j=\sum_{k=1}^3p_{j,k}$ with
%$p_{j,k}=%\sum\limits_{\substack{s\in \{0,1,2,3\}^{|J^c|}\\\mathrm{supp}s\subseteq J^c}}
%\|g_{j,k}\|_2^2$ and %$q_j=\sum_{k=1}^3q_{j,k}$ with
%$q_{j,k}=%\sum\limits_{\substack{s\in \{0,1,2,3\}^{|J^c|}\\\mathrm{supp}s\subseteq J^c}}
%\|g_{j,k}\|_1$.
\begin{lem}\label{L2} Let $A\in M_2(\C)^{\otimes n}$ with $\|A\|\leq1$ and $J\subseteq[n]$.
	For each $j\in J$ and $k\in\{1,2,3\}$,
	\begin{itemize}
		\item[(i)]  $\|g_{j,k}\|_2^2\leq \|g_{j,k}\|_1\leq\|g_{j,k}\|_2$; %As a consequence, $p_{j}\lesssim q_{j}\lesssim\sqrt{p_{j}}$.
		%	\item[(ii)] $\|g_{j,k}\|_2^2\leq\mathrm{Inf}_j[A]$; %As a consequence, $p_j\lesssim\mathrm{Inf}_j[A]$;
		%	\item [(iii)] $\|g_{j,k}\|_1\leq\mathrm{Inf}^1_j[A]$. %As a consequence, $\sum_{j\in J}q_j\leqm\mathrm{Inf}^1[A]$.
		\item [(ii)] $\|g_{j,k}\|_p^p\leq\mathrm{Inf}^p_j[A]$ for any $1\leq p\leq2$.
	\end{itemize}
\end{lem}
\begin{proof}
	(i) Since $\|A\|\leq1$, by Lemma \ref{restrict} and contractivity of conditional expectation $E_J$, one has
	$$\|g_{j,k}\|=\big\|E_J(A \, \sigma_{k(j)})\big\|\leq\|A \sigma_{k(j)}\|\leq \|A\|\leq1.$$
	%$A$ is bounded,
	%where we also used the fact that the partial trace $\frac{1}{2^{|J|}}\mathrm{tr}_J$ is a contraction.
	%	we have $\|g_{j,k}\|\lesssim1$ for all $j\in J$ and $k\in\{1,2,3\}$.
	It then follows from  H\"older's inequality that
	$$\|g_{j,k}\|_2^2\leq\|g_{j,k}\| \, \|g_{j,k}\|_1\leq\|g_{j,k}\|_1\le \|g_{j,k}\|_2.$$
	
	(ii) Fix $j\in J$ and $k\in\{1,2,3\}$. %We claim that
	%	\begin{align}\label{ress}
		%	g_j=\frac{1}{2^{|J|}}\mathrm{tr}_J(A \sigma_{\{j\}})=\frac{1}{2^{|J|}}\mathrm{tr}_J(d_j(A) \sigma_{\{j\}}),
		%\end{align}
		%where $\frac{1}{2^{|J|}}\mathrm{tr}_J$ denotes the partial trace restricted to the subset $J$ and $\sigma_{\{j\}}=\sigma_{\{0,\cdot\cdot\cdot,s_j,\cdot\cdot,0\}}$ with $s_j\in\{1,2,3\}$.
		%Indeed, the first equality just follows from the definition.
		%For the second, by
		By Lemma \ref{restrict} and the Cauchy-Schwarz inequality, one has
		%Now	applying (\ref{ress}) and H\"older's inequality twice, we get
		\begin{align*}
			\|g_{j,k}\|_2^2
			= &\  \tau_{J^c}\Big(\Big|\tau_J\big(d_j(A) \sigma_{k(j)}\big)\Big|^2\Big)\\
			\le &\  \tau_{J^c}\tau_J\big(\big|d_j(A) \sigma_{k(j)}\big|^2\big)
			= \tau( |d_j(A)  \, \sigma_{k(j)}|^2) \leq\|d_j(A)\|_2^2,
		\end{align*}	
		where the first inequality used the Kadison-Schwarz inequality $\Phi(A^*A)\ge \Phi(A)^*\Phi(A)$ for completely positive unital map $\Phi=\tau_J$, and
		the last equality follows because the complemented (normalized) `partial traces' combine to the (normalized) total trace $\tau_J\tau_{J^c}=\tau_{[n]}$.
		
		%	$$\|g_j\|_2^2=\frac{1}{2^{|J^c|}}\mathrm{tr}_{J^c}\big(\big|\frac{1}{2^{|J|}}\mathrm{tr}_J(d_j(A) \sigma_{\{j\}})\big|^2\big)\leq\frac{1}{2^{n}}\mathrm{tr}(|d_j(A) \sigma_{j}|)\leq\|d_j(A)\|_1,$$
		
		The above argument also works for $\|g_{j,k}\|_1$. Indeed,
		by Lemma \ref{restrict} and H\"older's inequality, we have
		$$\|g_{j,k}\|_1=\tau_{J^c}\Big(\Big|\tau_{J}\big(d_j(A) \,  \sigma_{k(j)}\big)\Big|\Big)\leq\tau(|d_j(A)  \,  \sigma_{k(j)}|)\leq\|d_j(A)\|_1.$$
		The case $1<p<2$ can be obtained by complex interpolation.
		%(iv) \LG{Do we have
			%		\begin{align*}
				%%\|g_{j,k}\|_p^p
				%	= &\  \tau_{J^c}\Big(\Big|\tau_J\big(d_j(A) \sigma_{k(j)}\big)\Big|^p\Big)\\
				%	\le &\  \tau_{J^c}\tau_J\big(\big|d_j(A) \sigma_{k(j)}\big|^p\big)
				%%	= \tau( |d_j(A)  \, \sigma_{k(j)}|^p) \leq\|d_j(A)\|_p^p
				%\end{align*}	
				%??}
		\end{proof}
		\begin{lem}\label{inter}
			%For each	$k\in\{1,2,3\}$,
			$\sum_{k=1}^3\sum_{j\in J}\mathrm{Inf}[%\sum\limits_{\substack{s\in \{0,1,2,3\}^{|J^c|}\\\mathrm{supp}s\subseteq J^c}}
			g_{j,k}]\leq\mathrm{Inf}[A]$.
		\end{lem}
		\begin{proof}Since $ g_{j,k}=\tau_J(A\sigma_{k(j)})\in M_2(\mathbb{C})^{J^c}$,
			$d_i(g_{j,k})=0$ for $i \in J$. Hence \[	\mathrm{Inf}_i[g_{j,k}] = \| d_i(g_{j,k}) \|_2^2 = 0 \ , \ \mathrm{Inf}[g_{j,k}]=\sum_{i\in J^c} \mathrm{Inf}_i[g_{j,k}]\ .\] It suffices to show for every $i\in J^c$
			\begin{align}\label{pare}
				\sum_{k=1}^3\sum_{j\in J}\mathrm{Inf}_i[%\sum\limits_{\substack{s\in \{0,1,2,3\}^{|J^c|}\\\mathrm{supp}s\subseteq J^c}}
				g_{j,k}]\leq	\mathrm{Inf}_i[A].
			\end{align}	
			Then it follows from \eqref{pare} and Fubini's theorem that
			$$\mathrm{Inf}[A]\geq\sum_{i\in J^c}\mathrm{Inf}_i[A]\geq\sum_{i\in J^c}\sum_{k=1}^3\sum_{j\in J}\mathrm{Inf}_i[%\sum\limits_{\substack{s\in \{0,1,2,3\}^{|J^c|}\\\mathrm{supp}s\subseteq J^c}}
			g_{j,k}]=\sum_{k=1}^3\sum_{j\in J}\mathrm{Inf}[%\sum\limits_{\substack{s\in \{0,1,2,3\}^{|J^c|}\\\mathrm{supp}s\subseteq J^c}}
			g_{j,k}].$$
			To prove \eqref{pare}, fix $i\in J^c$. By Parseval's identity, one has
			$$\mathrm{Inf}_i[A]=\sum_{\substack{s\in \{0,1,2,3\}^{J^c}\\s_i\neq0}}\sum_{t\in \{0,1,2,3\}^{J}}|\widehat A_{s\cup t}|^2,$$
			%	Note that the above has a lower bound just by considering only one coefficients of size one. That is
			which implies that
			\begin{align*}
				\mathrm{Inf}_i[A]\geq &\  \sum_{\substack{s\in \{0,1,2,3\}^{J^c}\\s_i\neq0}}\sum_{\substack{t\in \{0,1,2,3\}^{J}\\|\mathrm{supp} \, t|=1}}|\widehat A_{s\cup t}|^2\\
				= &\ \sum_{\substack{s\in \{0,1,2,3\}^{J^c}\\s_i\neq0}}\sum_{k=1}^3\sum_{j\in J}|\widehat A_{s\cup k(j)}|^2\\
				=&\ \sum_{k=1}^3\sum_{j\in J}\sum_{\substack{s\in \{0,1,2,3\}^{J^c}\\s_i\neq0}}|\widehat A_{s\cup k(j)}|^2=\sum_{k=1}^3\sum_{j\in J}\|d_i(%\sum%\limits_{\substack{s\in \{0,1,2,3\}^{|J^c|}\\\mathrm{supp}s\subseteq J^c}}
				g_{j,k})\|_2^2.
			\end{align*}
			%	
			%	$$\mathrm{Inf}_i[A]\geq\sum_{\substack{s\in \{0,1,2,3\}^{J^c}\\s_i\neq0}}\sum_{\substack{t\in \{0,1,2,3\}^{J}\\|\mathrm{supp}t|=1}}|\widehat A_{\{j\}\cup s}|^2=\sum_{\substack{s\in \{0,1,2,3\}^{J^c}\\s_i\neq0}}\sum_{j\in J}\sum_{\substack{t\in \{0,1,2,3\}^{J}\\t_j\neq0}}|\widehat A_{s\cup t}|^2=\sum_{j\in J}\|d_i(%\sum%\limits_{\substack{s\in \{0,1,2,3\}^{|J^c|}\\\mathrm{supp}s\subseteq J^c}}
			%	g_{j})\|_2^2.$$
			This completes the proof.
		\end{proof}
		
		We are ready to prove Lemma \ref{bmo}.
		\begin{proof}[Proof of Lemma \ref{bmo}.]
			Note that $\|g_{j,k}\|\leq1$.   Then by Lemma \ref{inter} and Lemma \ref{log-sobo}, we have that %for each $k\in\{1,2,3\}$ there is an absolute constant $K>0$ such that
			%\begin{align*}
			%\mathrm{Inf}(A)
			%	\geq &\  \sum_{j\in J}\mathrm{Inf}\Big(\sum\limits_{\substack{s\in \{0,1,2,3\}^{|J^c|}\\\mathrm{supp}s\subseteq J^c}}g_{j,s}\Big)\\
			%	\gtrsim &\ \sum_{j\in J}\Big\|\sum\limits_{\substack{s\in \{0,1,2,3\}^{|J^c|}\\\mathrm{supp}s\subseteq J^c}}g_{j,s}\Big\|_2^2\log\Big\|\sum\limits_{\substack{s\in \{0,1,2,3\}^{|J^c|}\\\mathrm{supp}s\subseteq J^c}}g_{j,s}\Big\|_2^2.
			%\end{align*}
			\begin{align}\mathrm{Inf}[A]\geq\sum_{k=1}^3\sum_{j\in J}\mathrm{Inf}[g_{j,k}]\geq\sum_{k=1}^3\sum_{j\in J}\Big(\frac{1}{2}\norm{g_{j,k}}{2}^2\log\frac{1}{ \norm{g_{j,k}}{2}^2}-\norm{g_{j,k}}{1}^{1/2}\norm{g_{j,k}}{2}\Big). \label{eq:later}
			\end{align}
			%where $f_j=\sum\limits_{\substack{s\in \{0,1,2,3\}^{|J^c|}\\\mathrm{supp}s\subseteq J^c}}g_{j,s}$.
			%Note that by Parseval theorem and Minkowski's inequality,
			%$$\|f_j\|_2^2=p_j\ \ \mbox{and}\ \ \|f_j\|_1\leq q_j.$$
			It further follows from Cauchy-Schwarz inequality that
			\begin{align*}
				\mathrm{Inf}[A]
				\geq &\  \frac{1}{2}\sum_{k=1}^3\sum_{j\in J}\norm{g_{j,k}}{2}^2\log\frac{1}{ \norm{g_{j,k}}{2}^2}-\sqrt{\sum_{k=1}^3\sum_{j\in J}\norm{g_{j,k}}{1}}\sqrt{\sum_{k=1}^3\sum_{j\in J}\norm{g_{j,k}}{2}^2}.%\\
				%\geq &\ -\sum_{j\in J}p_j\log p_j-K\sqrt{\sum_{j\in J}q_j}\sqrt{\sum_{j\in J}p_j}.
			\end{align*}
			%$$\mathrm{Inf}(A)\gtrsim-\sum_{j\in J}p_j\log p_j-K\sqrt{\sum_{j\in J}\|f_j\|_1}\sqrt{\sum_{j\in J}\|f_j\|^2_2}\geq-\sum_{j\in J}p_j\log p_j-K\sqrt{\sum_{j\in J}\|f_j\|_1}\sqrt{\sum_{j\in J}\|f_j\|^2_2}.$$
			By Lemma \ref{L2} (ii), $\sum_{j\in J}\norm{g_{j,k}}{1}\leq\mathrm{Inf}^1[A]$ for all $k=1,2,3$.  On the other hand, recall that $W_{J}(A)=\sum_{k=1}^3\sum_{j\in J}{\norm{g_{j,k}}{2}^2}$. Therefore,
			$$\mathrm{Inf}[A]\geq\frac{1}{2}\sum_{k=1}^3\sum_{j\in J}{\norm{g_{j,k}}{2}^2}\log\frac{1}{ \norm{g_{j,k}}{2}^2}-\sqrt{3}\sqrt{\mathrm{Inf}^1[A]}\sqrt{W_{J}(A)}.$$
			To estimate the first term on the right-hand side above,
			we conclude from Lemma \ref{L2} (ii) and Lemma \ref{was21}  that for each  $j\in J$ and $k\in\{1,2,3\}$
			$${\norm{g_{j,k}}{2}^2}\leq\mathrm{Inf}_j[A]\leq2 \mathrm{Inf}^1_j[A].$$
			We finally find
			$$\mathrm{Inf}[A]\geq \frac{1}{2}\log\Big(\frac{1}{\max_{j\in J}2\mathrm{Inf}_j^1[A]}\Big)\cdot W_{J}(A)-\sqrt{3}\sqrt{\mathrm{Inf}^1 [A]}\cdot\sqrt{W_{J}(A)}.$$
			This finishes the proof.
		\end{proof}

\subsection{Soft Chunks}
Another important tool in proving Theorem \ref{thm:main1} is the chunk operator, which is defined as the difference of two noise operators. %as follows. Recall that the noise operator $T_\delta$ on $A$ is defined in (\ref{noise}) with $\delta\in[0,1]$.
\begin{defi}\label{soft}
	Let $A\in M_2(\C)^{\otimes n}$ with $\|A\|\leq1$ and let $d=2^l$ with some non-negative integer $l\in \mathbb{N}_0$. The {\em soft chunk} of $A$ of degree $d$ is the operator $B_d\in M_2(\C)^{\otimes n}$ given by
	$$B_d=\big(T_{1-\frac{1}{3d}}-T_{1-\frac{1}{2d}}\big)A.$$
\end{defi}
We summarize the following useful properties of soft chunks.
\begin{lem}\label{keylemma}
	Let $A\in M_2(\C)^{\otimes n}$ with $\|A\|\leq1$, $j\in[n]$ and let $d=2^l$ with some non-negative integer $l\in \mathbb{N}_0$. Let $B_d\in M_2(\C)^{\otimes n}$ be the soft chunk operator of $A$. Then the following statements hold.
	\begin{itemize}
		\item[(i)]  $\|B_d\|\leq2$ %
		and $\mathrm{Inf}_j^p[B_d]\leq2^p\mathrm{Inf}_j^p[A]$ for any $1\leq p<\infty$. As a consequence, $\mathrm{Inf}^p[B_d]\leq2^p\mathrm{Inf}^p[A]$ for any $1\leq p<\infty$.
		\item[(ii)] For any $s\in \{0,1,2,3\}^{n}$ with $d\leq|\mathrm{supp} \, s|<2d$, we have $\frac{1}{20}|\widehat A_s|\leq|\widehat{(B_d)}_s|\leq|\widehat A_s|$. As a consequence, $\frac{1}{400}W_{\approx d}[A]\leq W_{\approx d}[B_d]\leq W_{\approx d}[A]$.
		%$\|B_d\|_2^2\geq W_{\approx d}[B_d]\approx W_{\approx d}[A]$.
		%,\ \sum_d\mathrm{Inf}_j(B_d)\gtrsim\mathrm{Inf}_j(A),\ \sum_d\mathrm{Inf}(B_d)\gtrsim\mathrm{Inf}(A).$$
		\item[(iii)] $ \sum_{d}\mathrm{Inf}_j[B_d]\leq\mathrm{Inf}_j[A]$. As a consequence, $\sum_d\mathrm{Inf}[B_d]\leq\mathrm{Inf}[A]$ where the $d$ is summing over all non-negative powers of $2$. %(All $d$ here are powers  of 2.)
		%\sum_d\|B_d\|_2^2\lesssim\mathrm{var}(A),\ \sum_d\mathrm{Inf}(B_d)\lesssim\mathrm{Inf}(A).$$
		%	\item[(iv)] \LG{$\mathrm{Inf}_j^p[B_d]\leq2^p\mathrm{Inf}_j^p[A]$ for any $1\leq p<2$. As a consequence, $\mathrm{Inf}^p[B_d]\leq2^p\mathrm{Inf}^p[A]$.}
	\end{itemize}
\end{lem}
\begin{proof} %We shall assume that $d \neq 1$, as the case $d = 1$  is similar but easier.
	(i) The boundedness of $B_d$ follows from contractivity of the noise operator $T_\delta$ and $\|A\|\leq1$. By the intertwining relation $T_\delta d_j=d_j T_\delta $,
	\begin{align*}
		\mathrm{Inf}_j^p[B_d]
		\leq &\  \Big(\|d_jT_{1-\frac{1}{3d}}(A)\|_p+\|d_jT_{1-\frac{1}{2d}}(A)\|_p\Big)^p\\
		= &\ \Big(\|T_{1-\frac{1}{3d}}d_j(A)\|_p+\|T_{1-\frac{1}{2d}}d_j(A)\|_p\Big)^p\leq2^p\mathrm{Inf}_j^p[A].
	\end{align*}	
	
	(ii) Note that by definition
	\begin{align}\label{up}
		\widehat{(B_d)}_s=\Big(\Big(1-\frac{1}{3d}\Big)^{|\mathrm{supp} \, s|}-\Big(1-\frac{1}{2d}\Big)^{|\mathrm{supp} \, s|}\Big)\widehat A_s.
	\end{align}
	Clearly, $|\widehat{(B_d)}_s|\leq|\widehat A_s|$. On the other hand,  set $\varphi(t)=\big(1-1/(3d)\big)^t-\big(1-1/(2d)\big)^t$ with $t\in[d,2d)$. Since $x\mapsto (1+1/x)^x$ is increasing on $(0,\infty)$, one has for any $t\in[d,2d)$,
	$$\frac{(1-1/(3d))^t}{(1-1/(2d))^t}\geq\Big(\frac{6d-2}{6d-3}\Big)^d\geq \Big(1+\frac{1}{6d-3}\Big)^{\frac{6d-3}{6}}\geq\frac{2\sqrt{3}}{3},$$
	which implies that
	$$\varphi(t)\geq(1-\frac{\sqrt{3}}{2})\Big(1-\frac{1}{3d}\Big)^t>(1-\frac{\sqrt{3}}{2})\Big(1-\frac{1}{3d}\Big)^{2d}\geq\frac{4}{9}(1-\frac{\sqrt{3}}{2})>\frac{1}{20}.$$
	
	%Since  $ e^{-a}\le (1 - \frac{a}{d})^d \le 1 $ for $a\in\R_+$, %(and similarly for $\sqrt{e})$
	%it is easy to check that the multiplicative factor in front of $\widehat A_s$ is proportional to an absolute constant %(involving $\frac{1}{e}$)
	%whenever $d\leq|\mathrm{supp} \, s|<2d$. %Then the conclusions are easy to verify.
	
	(iii) %By (\ref{up}),  $\widehat{(B_d)}_\emptyset=0$.  Hence,
	By Parseval's identity and Fubini's theorem, for each $j\in [n]$,
	\begin{align*}
		\sum_{d=2^l, l\in \mathbb{N}_0}\mathrm{Inf}_j[B_d]
		= &\  \sum_{d=2^l, l\in \mathbb{N}_0}\sum_{\substack{s\in \{0,1,2,3\}^{n}\\s_j\neq0}}|\widehat{(B_d)}_s|^2\\
		= &\ \sum_{\substack{s\in \{0,1,2,3\}^{n}\\s_j\neq0}}\sum_{d=2^l, l\in \mathbb{N}_0}\Big(\Big(1-\frac{1}{3d}\Big)^{|\mathrm{supp} \, s|}-\Big(1-\frac{1}{2d}\Big)^{|\mathrm{supp} \, s|}\Big)^2|\widehat A_s|^2.
	\end{align*}
	As $d$ ranges only over powers of 2,
	\begin{align*}&\sum_{l\in \mathbb{N}_0}\Big(\Big(1-\frac{1}{3\cdot 2^l}\Big)^{|\mathrm{supp} \, s|}-\Big(1-\frac{1}{2\cdot 2^l}\Big)^{|\mathrm{supp} \, s|}\Big)^2\\ \leq &\sum_{l\in \mathbb{N}_0}\Big(\Big(1-\frac{1}{3\cdot 2^l}\Big)^{|\mathrm{supp} \, s|}-\Big(1-\frac{1}{2\cdot 2^l}\Big)^{|\mathrm{supp} \, s|}\Big)\le\Big(\frac{2}{3}\Big)^{ |\mathrm{supp} \, s|}\le 1.\end{align*}
	%$$\sum_d\|B_d\|_2^2=\sum_d\sum_{\substack{s\in \{0,1,2,3\}^{n}\\s\neq0}}|\widehat{(B_d)}_s|^2=\sum_{\substack{s\in \{0,1,2,3\}^{n}\\s\neq0}}\sum_d.$$
	%The argument works for the second inequality involves Influence. % is similar just by noting $\mathrm{Inf}(A)=\sum_{s\in \{0,1,2,3\}^{n}}|s||\widehat A_s|^2$.
	This finishes the proof.
\end{proof}

\begin{lem}\label{po}
	Let $A\in M_2(\C)^{\otimes n}$ with $\|A\|\leq1$. Define
	$$D_{\mathrm{good}}=\Big\{d=2^l,l\in\N_0\ \big|\  W_{\approx d}[A]\geq\frac{\mathrm{Var}[A]^2}{16\mathrm{Inf}^1[A]}\Big\}.$$
	Then %there is an absolute constant $C>0$ such that
	$$\mathrm{Var}[A]\leq 4\sum_{d\in D_{\mathrm{good}}}W_{\approx d}[A].$$
\end{lem}
\begin{proof}. Since $\mathrm{Var}[A]=\sum_{s\neq 0}|\widehat A_s|^2$, we have
	%Observe that
	$$\mathrm{Var}[A]=\sum_{d=2^l,l\in\N_0} W_{\approx d}[A]=\sum_{d\in D_{\mathrm{good}}}W_{\approx d}[A]+\sum_{d\notin D_{\mathrm{good}}}W_{\approx d}[A].$$
	Hence, it is enough to show there is a constant $0<K<1$ such that
	\begin{align}\label{u}
		\sum_{d\notin D_{\mathrm{good}}}W_{\approx d}[A]\leq K \, \mathrm{Var}[A].
	\end{align}
	All $d$'s below are powers of $2$.
	To show \eqref{u}, we consider two cases: if $d\leq M:=\frac{4\mathrm{Inf}^1[A]}{\mathrm{Var}[A]}$, it follows from $d\notin D_{\mathrm{good}}$ that
	\begin{align}\label{u1}
		\sum_{d:d\notin D_{\mathrm{good}},d\leq M}W_{\approx d}[A]\leq\sum_{d:d\notin D_{\mathrm{good}},d\leq M}
		\frac{\mathrm{Var}[A]^2}{16\mathrm{Inf}^1[A]}\le \frac{4\mathrm{Inf}^1[A]}{\mathrm{Var}[A]}\cdot \frac{\mathrm{Var}[A]^2}{16\mathrm{Inf}^1[A]}=\frac{1}{4}\mathrm{Var}[A].\end{align}
	%\LG{This inequality is very loose, only $\log M$ term in the summation.}
	For the case $d\geq M$, since by (\ref{rwz25}) we have
	$$\mathrm{Inf}[A]\geq\sum_{d=2^l\geq M}\sum_{d\leq|\mathrm{supp} \, s|<2d}|\mathrm{supp} \, s||\widehat A_s|^2\geq M\sum_{d=2^l\geq M}W_{\approx d}[A]$$
	and using Lemma \ref{was21}
	%d\ref{inf},
	we conclude that
	\begin{align}\label{u2}
		\sum_{d\geq M}W_{\approx d}[A]\leq\frac{\mathrm{Inf}[A]}{M}=\mathrm{Var}[A]\frac{\mathrm{Inf}[A]}{4\mathrm{Inf}^1[A]}\leq\frac{1}{2}\mathrm{Var}[A].
	\end{align}
	Combining (\ref{u1}) with (\ref{u2}), we obtain (\ref{u}) with $K=\frac{3}{4}$.
	This completes the proof.
\end{proof}

We shall now prove Theorem \ref{thm:main1}.
\begin{proof}[Proof of Theorem \ref{thm:main1}.]
%Let $C>0$ be  a constant  that will be chosen later.
We argue  by contradiction.   Let $C=180000$. Assume that
%large enough,
%there is a constant $C$ large enough such that
\begin{align}\label{e121}
\mathrm{Inf}_j^1[A]\leq2^{-C\frac{\mathrm{Inf}^1[A]}{\mathrm{Var}[A]}},
\end{align}  for all coordinates $j\in[n]$.
%where $C>0$ is an absolute constant large enough which will be determined later.

Let $B_d$ be the soft chunk of $A$ of degree $d$  with $d\in D_{\mathrm{good}}$. %Let $J$ be a subset of $[n]$.
By Lemma \ref{keylemma} (i),  $\|B_d\|\leq2$.
Then   applying Lemma \ref{bmo} to $B_d/2$, we see that for any subset $J\subseteq[n]$ %and any $k\in\{1,2,3\}$
%Let $B_d$ be defined in Definition \ref{soft}, where $d\in D_{\mathrm{good}}$. Consider $\frac{1}{d}$-random restriction and let $J\subseteq[n]$ be such that (\ref{e}) holds. Applying Lemma \ref{bmo} to $B_d$, we get
\begin{align}\label{e11}
\mathrm{Inf}[B_d]\geq\frac{1}{2} \log\Big(\frac{1}{\max_{j\in J}\mathrm{Inf}_j^1[B_d]}\Big)\cdot W_J(B_d)-\sqrt{6}\sqrt{\mathrm{Inf}^1 [B_d]}\cdot\sqrt{W_J(B_d)}
\end{align}
with $W_J(B_d)$ given by
$$W_J(B_d)=\sum_{\substack{v\in\{0,1,2,3\}^{n}\\ |\mathrm{supp}\, v\cap J|=1}} |\widehat {(B_d)}_{v}|^2.$$
In particular, for any $\frac{1}{2d}$-random subset $J$ of $[n]$, we have  from
Lemma \ref{first level} that %for each $k\in\{1,2,3\}$
\begin{align*}
\mathbb	E_J[W_J(B_d)]
	= &\  \sum_{v\in \{0,1,2,3\}^{n}} |\widehat {(B_d)}_v|^2  \, \mathrm{Pr}[|J\cap \mathrm{supp} \, v|=1]\\
	\geq &\ \sum_{d\leq|\mathrm{supp} \, v|<2d} | \widehat {(B_d)}_v|^2  \, \mathrm{Pr}[|J\cap \mathrm{supp} \, v|=1].
\end{align*}
%\begin{align*}%\label{e11}
%\mathbb	E_J[W_{J,k}%\sum_{\substack{t\in \{0,1,2,3\}^{J}\\|\mathrm{supp}t|=1}}\sum_{s\in \{0,1,2,3\}^{J^c}} |\widehat A_{t\cup s}|^2
%]\geq\sum_{d\leq|\mathrm{supp}v|<2d} |\widehat {(B_d)}_v|^2\mathrm{Pr}[|J\cap \mathrm{supp}v|=1]\gtrsim W_{\approx d}[B_d]\approx W_{\approx d}[A],
%\end{align*}
As we showed in Lemma \ref{first level}, %for $d \geq 2$,
%ote that by definition \ref{randoma},
\begin{align}\label{e19}
\mathrm{Pr}[|J\cap \mathrm{supp} \, v|=1]=\frac{|\mathrm{supp} \, v|}{2d}(1-\frac{1}{2d})^{|\mathrm{supp} \, v|-1}
\end{align}
whenever $d\leq|\mathrm{supp} \, v|<2d$.
For such $d$ we have
$$\frac{|\mathrm{supp} \, v|}{2d}(1-\frac{1}{2d})^{|\mathrm{supp} \, v|-1} \geq \frac{1}{2}(1-\frac{1}{2d})^{2d-1} \geq \frac{1}{2e}\geq \frac{1}{6}.$$ %for $d \geq 2$.
%For $d = 1$ ADD
Hence,
$$\mathbb E_J[W_J(B_d)]\geq\frac{1}{6} W_{\approx d}[B_d]\geq\frac{1}{2400} W_{\approx d}[A]\geq\frac{\mathrm{Var}[A]^2}{38400\mathrm{Inf}^1[A]},$$
where %in the second inequality we used the fact that for any $v$ with $d\leq|\mathrm{supp}v|<2d$, the probability it intersects $J$ in a single element is an absolute constant and
the second inequality is a consequence of Lemma \ref{keylemma} (ii), and the last inequality is due to $d\in D_{\mathrm{good}}$.

In the rest, we fix some $J\subseteq[n]$ such that
\begin{align}\label{e}
	%	\sum_{\substack{t\in \{0,1,2,3\}^{J}\\|\mathrm{supp}t|=1}}\sum_{s\in \{0,1,2,3\}^{J^c}} |\widehat A_{t\cup s}|^2
	W_J(B_d)\geq\frac{1}{2400} W_{\approx d}[A]\geq\frac{\mathrm{Var}[A]^2}{38400\mathrm{Inf}^1[A]}.
\end{align}

%Since $d\in D_{\mathrm{good}}$, it follows from  Lemma \ref{keylemma}(ii) %and Lemma \ref{po}
%that
%$$W\gtrsim W_{\approx d}[B_d]\gtrsim W_{\approx d}[A]\geq\frac{\mathrm{var}(A)^2}{16\mathrm{Inf}^1(A)}.$$
By Lemma \ref{keylemma} (i) and the assumption (\ref{e121}), one can deduce from (\ref{e11}) that
\begin{align}\label{eq:est}
	\mathrm{Inf}[B_d]
	\geq &\  \frac{1}{2}\log\Big(\frac{1}{\max_{j\in J}2\mathrm{Inf}_j^1[A]}\Big)\cdot W_J(B_d)-2\sqrt{3}\sqrt{\mathrm{Inf}^1 [A]}\cdot\sqrt{W_J(B_d)}\nonumber\\
	\geq &\ \frac{1}{4}\Big(C  \,  \frac{\mathrm{Inf}^1[A]}{\mathrm{Var}[A]}-1\Big)\cdot W_J(B_d)-2\sqrt{3}\sqrt{\mathrm{Inf}^1 [A]}\cdot\sqrt{W_J(B_d)}.
\end{align}
By (\ref{eq:poincare2}) and Lemma \ref{was21}, $\mathrm{Var}[A]\leq\mathrm{Inf}[A]\leq2\mathrm{Inf}^1[A]$. This implies by noting $C=180000$ that
$$\frac{1}{4}\Big(C  \,  \frac{\mathrm{Inf}^1[A]}{\mathrm{Var}[A]}-1\Big)\geq40000\frac{\mathrm{Inf}^1[A]}{\mathrm{Var}[A]}$$
and substituting this into (\ref{eq:est}), one gets
\begin{align}\label{eq:est2}
	\mathrm{Inf}[B_d]
	\geq  40000\frac{\mathrm{Inf}^1[A]}{\mathrm{Var}[A]}\cdot W_J(B_d)-2\sqrt{3}\sqrt{\mathrm{Inf}^1 [A]}\cdot\sqrt{W_J(B_d)}.
\end{align}
Combining with the second inequality in (\ref{e}), we have
\begin{align*}%\label{eq:est1}
	\mathrm{Inf}[B_d]
	\geq &\ \frac{\mathrm{Inf}^1[A]}{\mathrm{Var}[A]} W_J(B_d)\Big(40000-2\sqrt{3}\frac{\mathrm{Var}[A]}{\sqrt{\mathrm{Inf}^1 [A]}\cdot\sqrt{W_J(B_d)}}\Big), \\ \ge & \frac{\mathrm{Inf}^1[A]}{\mathrm{Var}[A]} W_J(B_d)(40000-\sqrt{460800})\\ >&36000\frac{\mathrm{Inf}^1[A]}{\mathrm{Var}[A]} W_J(B_d).
\end{align*}
%It is not difficult to check from (\ref{e}) that the first term dominates the second in the last line above by choosing a large enough constant $C$. Indeed, let $C_0>0$ be an absolute constant (independent of $A, d, J,n$) in (\ref{e}) such that
%\[W_J(B_d)\ge C_0 \frac{\mathrm{Var}[A]^2}{\mathrm{Inf}^1[A]}\ .\]
%Choose $C=C'+\frac{K}{\sqrt{C_0}}$, we have by \eqref{eq:est} that
%$$\mathrm{Inf}[B_d]\gtrsim C\frac{\mathrm{Inf}^1[A]}{\mathrm{Var}[A]}\cdot W_J-K\sqrt{\mathrm{Inf}^1 [A]}\cdot\sqrt{W_J}.$$
Now we use the first inequality stated in (\ref{e}) to get
$$\mathrm{Inf}[B_d]\geq \,  36000\frac{\mathrm{Inf}^1[A]}{\mathrm{Var}[A]}\cdot W_J(B_d)\geq \,  15 \frac{\mathrm{Inf}^1[A]}{\mathrm{Var}[A]}\cdot W_{\approx d}[A].$$
By applying Lemma \ref{was21}, Lemma \ref{keylemma} (iii) and summing over $d\in D_{\mathrm{good}}$, we find
\begin{align*}\mathrm{Inf}^1[A]\geq\frac{1}{2}\mathrm{Inf}[A]\geq \frac{1}{2}\sum_{d\in D_{\mathrm{good}}}\mathrm{Inf}[B_d]&\geq  \,  \frac{15}{2}\frac{\mathrm{Inf}^1[A]}{\mathrm{Var}[A]}\sum_{d\in D_{\mathrm{good}}}W_{\approx d}[A],
	\\ &\geq \frac{15}{8}\frac{\mathrm{Inf}^1[A]}{\mathrm{Var}[A]} \, \mathrm{Var}[A]=\frac{15}{8}\mathrm{Inf}^1[A],
\end{align*}
where the last inequality follows from
Lemma \ref{po}. This gives a contradiction. Therefore, there exists $j\in[n]$ such that (\ref{e121}) fails, which completes the proof.
\end{proof}
%\section{Proof of  Lemma \ref{bmo}}
Theorem \ref{thm:main1} reproduces quantum KKL theorem for $L^1$-influence as a corollary.
\begin{cor}
\label{main2}
	Let $A\in M_2(\C)^{\otimes n}$ with $\|A\|\leq1$. Then there exist an absolute constant $C>0$  such that one of the following two conclusions holds:
\begin{itemize}
\item[i)] there exists $j\in [n]$ such that $\mathrm{Inf}_{j}^1[A]\ge \frac{1}{\sqrt{n}}$.
\item[ii)] $\mathrm{Inf}^1[A]\geq C \log n {\mathrm{Var}[A]}$
\end{itemize}
\end{cor}
\begin{proof} %It is easy to see that Theorem \ref{main1} implies the following statement below.
%The proof of Theorem \ref{main1} depends on the property of random subset and Log-Sobolev inequality-Proposition \ref{log-sobolev}, while the argument in \cite[Theorem 3.8]{RWZ}
%Let us prove Theorem \ref{main2}. Indeed,
Let $C$ be the constant in Theorem \ref{thm:main1}.
If $\mathrm{Inf}^1[A]\geq \frac{\log\,  n}{2C}\mathrm{Var}[A]$, we are done. Otherwise,
%if $\mathrm{Inf}^1(A)\leq \frac{\log n}{2C}\mathrm{var}(A)$, then
Theorem \ref{thm:main1} implies that there is $j\in[n]$ such that $\mathrm{Inf}_j^1[A]\geq2^{-C\frac{\mathrm{Inf}^1[A]}{\mathrm{Var}[A]}}\geq\frac{1}{\sqrt{n}}$. The proof is complete.
\end{proof}
\begin{remark}\label{quantumkkl}
	For any $A\in M_2(\C)^{\otimes n}$ with $\|A\|\leq1$,  Corollary \ref{main2} clearly implies that there exists an absolute constant $C>0$ such that
	$$\max_{j\in[n]}\mathrm{Inf}_j^1[A]\geq C \, \mathrm{Var}[A]  \,  \frac{\log n \, }{n}.$$
Theorem \ref{thm:main1} therefore is stronger than the $L^1$-KKL Theorem for quantum Boolean function obtained in \cite[Theorem 3.9]{RWZ}.	
\end{remark}

\subsection{A quantum Talagrand's inequality}
A major tool in the proof of Rouz\'e, Wirth and Zhang \cite{RWZ} for the quantum KKL theorem is a quantum analog of Talagrand inequality \cite[Theorem 3.6]{RWZ}. Their approach used complex interpolation and semigroup techniques. In this subsection, we present an alternate proof to the quantum $L^1$-Talagrand inequality. Our approach is based on the random restrictions and the soft chunks used in previous subsection.
%\LG{``We  believe that our approach is quite powerful and may be expected to open up research on other problems in Boolean analysis, such as the Fourier-entropy conjecture \cite{KKLMS} and junta theorems
%\cite{Bou02,DFKO05}'' this sentence to be moved to introduction later}.
% of quantum  $L^1$-Talagrand's inequality.  Actually, the following theorem has already obtained by
%\textcolor{red}{Li: It would be better to return to simply $p=1$ case and hide the $1<p<2$ case in a remark and mentioned $p=2$ remains open. Now this result looks weaker than RWZ. Also the $\log^+$ can be replaced by $\log$}
The quantum Talagrand's inequality is stated as follows.
\begin{thm}\label{main3}
	Let $A\in M_2(\C)^{\otimes n}$ with $\|A\|\leq1$. Then  there exists an absolute constant $C>0$ such that $$\mathrm{Var}[A]\leq C\sum_{j\in[n]}\frac{\mathrm{Inf}^1_j[A]}{\log(1/\mathrm{Inf}^1_j[A])}.$$
	%	Here $\log^+(a)=\max\{ \log(a),0\}$ denotes the positive part of logarithmic function.
\end{thm}
We  start with the following two lemmas.
\begin{lem}\label{tala}
	Let $A\in M_2(\C)^{\otimes n}$ with $\|A\|\leq1$. %Assume that $\mathrm{Inf}_j^1[A]\leq1/e$ for all $j\in[n]$.
	Then one of the following two conclusions hold for all $d=2^l\geq1$ be a nonnegative integer power of 2:
	\begin{itemize}
		\item[(i)] $\displaystyle\sum_{j\in [n]} \, \frac{\mathrm{Inf}_j[A]}{\log(1/\mathrm{Inf}_j^1[A])}\geq\frac{1}{432}\frac{d(W_{\approx d}[A])^2}{\mathrm{Inf}[A]}$;
		\item[(ii)] $\displaystyle\sum_{j\in [n]}\frac{\mathrm{Inf}_j^1[A]}{\log(1/\mathrm{Inf}_j^1[A])}\geq \frac{1}{192}dW_{\approx d}[A]$.
	\end{itemize}
\end{lem}
\begin{proof}
	Let $J$ be  a  $\frac{1}{2d}$-random subset of $[n]$. For each $j\in J$ and $k\in\{1,2,3\}$, recall that
	$$g_{j,k}=\tau_J(A\sigma_{k(j)})=\sum\limits_{s\in \{0,1,2,3\}^{J^c}}%\sum\limits_{s\in \{0,1,2,3\}^{J^c}}
	\widehat A_{s\cup k(j)} \sigma_{s}$$
	and by \eqref{eq:later},
	%By applying Lemma \ref{inter} and Lemma \ref{log-sobo} as in \eqref{eq:later}, we can deduce that
	%In the proof of Lemma \ref{bmo}, we have \eqref{eq:later} that there is some absolute constant $K>0$ such that
	\begin{align}\label{e3}
		\mathrm{Inf}[A]\geq\frac{1}{2}\sum_{k=1}^3\sum_{j\in J}{\norm{g_{j,k}}{2}^2}\log (1/{\norm{g_{j,k}}{2}^2})-\sum_{k=1}^3\sum_{j\in J}\norm{g_{j,k}}{2}\norm{g_{j,k}}{1}^{1/2}.
	\end{align}
	In the following, we use the short notation $\sum_{k,j}$ for the summation $\sum_{k=1}^3\sum_{j\in J}$ over $g_{j,k}$. Recall that by the Poincar\'e inequality \eqref{eq:poincare2},
	\[ \mathrm{Inf}[A]\ge \mathrm{Var}[A]\ge \sum_{k,j}{\norm{g_{j,k}}{2}^2}\ .\]
	Combined with \eqref{e3}, we have
	\begin{align}\label{e2}
		\mathrm{Inf}[A]\geq\frac{1}{4}\sum_{k,j}{\norm{g_{j,k}}{2}^2}\big(1+\log({1}/{\norm{g_{j,k}}{2}^2})\big)-\sum_{k,j}\norm{g_{j,k}}{2}\norm{g_{j,k}}{1}^{1/2}.
	\end{align}
	%$W_J=\sum_{j\in J}p_j$
	We now divide the discussion into two cases depending on whether or not the first term after  taking the expectation over $J$ dominates  the second on the right hand side.\\
	
	\noindent \textbf{Case 1}: $\mathbb E_J[\frac{1}{4}\sum_{k,j}{\norm{g_{j,k}}{2}^2}\big(1+\log({1}/{\norm{g_{j,k}}{2}^2})\big)]\ge 2\mathbb E_J[\sum_{k,j}\norm{g_{j,k}}{2}{\|g_{j,k}\|^{1/2}_1}]$.
	
	In this case, %$\mathrm{Inf}[A]\gtrsim\mathbb E_J[\sum_{j\in J}p_j\log(1/p_j)]$.
	taking the expectation of $J$ on both sides of (\ref{e2}), we have
	$$\mathrm{Inf}[A]\geq\frac{1}{8}\mathbb  E_J\Big[\sum_{k,j}{\norm{g_{j,k}}{2}^2}\big(1+\log({1}/{\norm{g_{j,k}}{2}^2})\big)\Big].$$
	Then by the Cauchy-Schwarz inequality,
	\begin{align*}
		&\big(\mathbb E_J\big[\sum_{k,j}{\norm{g_{j,k}}{2}^2}\big]\big)^2
		\\ \leq &\ \mathbb E_J\Big[\sum_{k,j}\frac{{\norm{g_{j,k}}{2}^2}}{1+\log({1}/{\norm{g_{j,k}}{2}^2}\big)}\Big]\cdot\mathbb E_J\big[\sum_{k,j}{\norm{g_{j,k}}{2}^2}\big(1+\log({1}/{\norm{g_{j,k}}{2}^2})\big)\big]\\
		\leq &\ 8\mathbb E_J\Big[\sum_{k,j}\frac{{\norm{g_{j,k}}{2}^2}}{1+\log({1}/{\norm{g_{j,k}}{2}^2})}\Big]\cdot\mathrm{Inf}[A].
	\end{align*}
	On the other hand, using Lemma \ref{first level} and
	(\ref{e19}),
	\begin{align}\label{e11s}
		\mathbb E_J\big[\sum_{k,j}{\norm{g_{j,k}}{2}^2} \big]\geq\sum_{d\leq|\mathrm{supp} \, v|<2d} |\widehat A_v|^2 \, \mathrm{Pr}[|J\cap \mathrm{supp} \, v|=1]\geq\frac{1}{6}  W_{\approx d}[A].
	\end{align}
	Therefore,
	%Since from (\ref{e11}), $E_J[\sum_{j\in J}p_j]\gtrsim W_{\approx d}[A]$. It follows that
	\begin{align}\label{e11s1}
		\mathbb E_J\Big[\sum_{k,j}\frac{{\norm{g_{j,k}}{2}^2}}{1+\log({1}/{\norm{g_{j,k}}{2}^2})}\Big]\geq\frac{1}{288}\frac{(W_{\approx d}[A])^2}{\mathrm{Inf}[A]}.
	\end{align}
	Note that by Lemma \ref{L2} (ii) and (\ref{inf}), ${\norm{g_{j,k}}{2}^2}\leq \mathrm{Inf}_j[A]\leq2 \mathrm{Inf}_j^1[A]$ for each $j\in J$ and $k\in\{1,2,3\}$.
	Then replacing ${\norm{g_{j,k}}{2}^2}$ by $\mathrm{Inf}_j[A]$ in the numerator and $2 \mathrm{Inf}_j^1[A]$ in the denominator on the left-hand side of (\ref{e11s1}), one has
	$$\mathbb E_J\Big[\sum_{j\in J}\frac{3\mathrm{Inf}_j[A]}{1+\log(\frac{1}{2\mathrm{Inf}_j^1[A]})}\Big]
	= \mathbb E_J\Big[\sum_{j\in [n]} \frac{3\mathrm{Inf}_j[A]}{1+\log(\frac{1}{2\mathrm{Inf}_j^1[A]})}\; \chi_{J}(j) \Big]
	\geq\frac{1}{288}\frac{(W_{\approx d}[A])^2}{\mathrm{Inf}[A]},$$
	which gives by noting $J$ is a $\frac{1}{2d}$-random subset that
	$$ \sum_{j\in [n]} \, \frac{\mathrm{Inf}_j[A]}{\log({1}/\mathrm{Inf}_j^1[A])}\geq\frac{1}{432}\frac{d(W_{\approx d}[A])^2}{\mathrm{Inf}[A]}.$$
	%\LG{One can also have $\mathrm{Inf}_j[A]$ in the denominator.}
	This proves (i).\\
	%$\mathbb E_J[\frac{1}{4}\sum_{k,j}{\norm{g_{j,k}}{2}^2}\big(1+\log^+({1}/{\norm{g_{j,k}}{2}^2})\big)]\ge 2\mathbb E_J[\frac{1}{2-p}\sum_{k,j}\norm{g_{j,k}}{2}{\|g_{j,k}\|^{p/2}_p}]$
	
	\noindent \textbf{Case 2}: $2\mathbb E_J[\sum_{k,j}\norm{g_{j,k}}{2}{\|g_{j,k}\|^{1/2}_1}]>\mathbb E_J[\frac{1}{4}\sum_{k,j}{\norm{g_{j,k}}{2}^2}\big(1+\log({1}/{\norm{g_{j,k}}{2}^2})\big)]$.
	
	In this case, Cauchy-Schwarz inequality gives
	\begin{align}\label{eq:cs}
		&\big(\mathbb  E_J\big[\sum_{k,j}{\norm{g_{j,k}}{2}}{\norm{g_{j,k}}{1}^{1/2}}\big]\big)^2\nonumber
		\\	\leq & \mathbb E_J\Big[\sum_{k,j}\frac{{\norm{g_{j,k}}{1}}}{1+\log({1}/{\norm{g_{j,k}}{2}^2})}\Big]\cdot\mathbb E_J\big[\sum_{k,j}{\norm{g_{j,k}}{2}^2}(1+\log\big({1}/{\norm{g_{j,k}}{2}^2})\big)\big].
	\end{align}
	%\begin{align*}
	%	\Big(\mathbb E_J\big[\sum_{j\in J}{\norm{g_{j,k}}{2}^2}\log \frac{1}{{\norm{g_{j,k}}{2}^2}}\big]\Big)^2
	%\lesssim  &\ \big(\mathbb  E_J\big[\sum_{j\in J}\sqrt{{\norm{g_{j,k}}{2}^2}}\sqrt{{\norm{g_{j,k}}{1}}}\big]\big)^2
	%	\leq  \mathbb E_J\Big[\sum_{j\in J}\frac{{\norm{g_{j,k}}{1}}}{\log\frac{1}{{\norm{g_{j,k}}{2}^2}}}\Big]\cdot\mathbb E_J\big[\sum_{j\in J}{\norm{g_{j,k}}{2}^2}\log\frac{1}{{\norm{g_{j,k}}{2}^2}}\big].
	%\end{align*}
	%$$\big(E_J[\sum_{j\in J}p_j\log (1/p_j)]\big)^2
	%\lesssim E_J[K\sum_{j\in J}\sqrt{p_j}\sqrt{q_j}]\leq E_J[\sum_{j\in J}\frac{q_j}{\log(1/p_j)}] E_J[\sum_{j\in J}p_j\log(1/p_j)].$$
	Note that by assumption,
	$$\mathbb  E_J\big[\sum_{k,j}{\norm{g_{j,k}}{2}^2}(1+\log({1}/{\norm{g_{j,k}}{2}^2}))\big]\le 8\mathbb  E_J\Big[\sum_{k,j}{\norm{g_{j,k}}{2}}{\norm{g_{j,k}}{1}^{1/2}}\Big].$$
	Combining with \eqref{eq:cs}, we have
	%$$\mathbb E_J\Big[\sum_{k,j}\frac{{\norm{g_{j,k}}{p}^{p}}}{1+\log^+({1}/{\norm{g_{j,k}}{2}^2})}\Big]\geq\frac{2-p}{8}\mathbb  E_J\Big[\sum_{k,j}{\norm{g_{j,k}}{2}}{\norm{g_{j,k}}{p}^{p/2}}\Big].$$
	\begin{align*}\mathbb E_J\Big[\sum_{k,j}\frac{{\norm{g_{j,k}}{1}}}{1+\log({1}/{\norm{g_{j,k}}{2}^2})}\Big]\geq &\frac{1}{8}\mathbb  E_J\Big[\sum_{k,j}{\norm{g_{j,k}}{2}}{\norm{g_{j,k}}{1}^{1/2}}\Big]\\ \geq &\frac{1}{64}\mathbb E_J\big[\sum_{k,j}{\norm{g_{j,k}}{2}^2}\big(1+\log({1}/{\norm{g_{j,k}}{2}^2})\big)\big].\end{align*}
	%\[E_J\big[\sum_{k,j}{\norm{g_{j,k}}{2}^2}(1+\log^+({1}/{\norm{g_{j,k}}{2}^2}))\big]\le 4K^2 E_J\big[\sum_{k,j}{\norm{g_{j,k}}{2}^2}(1+\log^+({1}/{\norm{g_{j,k}}{2}^2}))\big]\]
	Thus it follows from (\ref{e11s}) that
	\begin{align*}W_{\approx d}[A]\leq6\mathbb E_J\big[\sum_{k,j}{\norm{g_{j,k}}{2}^2}\big]\leq &6\mathbb  E_J\big[\sum_{k,j}{\norm{g_{j,k}}{2}^2}\big(1+\log({1}/{\norm{g_{j,k}}{2}^2})\big)\big]\\ \leq &384\mathbb E_J\Big[\sum_{k,j}\frac{{\norm{g_{j,k}}{1}}}{1+\log({1}/{\norm{g_{j,k}}{2}^2})}\Big].\end{align*}
	%where the first inequality follows from (\ref{e11s}).
	%Since $E_J[\sum_{j\in J}p_j]\gtrsim W_{\approx d}[A]$,
	Again, it follows from Lemma \ref{L2} that ${\norm{g_{j,k}}{1}}$ are bounded  above by $\mathrm{Inf}^1_j[A]$ and ${\norm{g_{j,k}}{2}^2}\leq\mathrm{Inf}_j[A]\leq2\mathrm{Inf}^1_j[A]$ (Lemma \ref{was21}) for all $j\in[n]$ and all $k\in\{1,2,3\}$; moreover,
	each coordinate appears in $J$ with probability $\frac{1}{2d}$,  the similar argument at the end of Case 1 yields
	$$\sum_{j\in [n]}\frac{\mathrm{Inf}_j[A]}{\log(1/\mathrm{Inf}_j^1[A])}\geq \frac{1}{192}dW_{\approx d}[A]. $$
	This completes the proof.
\end{proof}
The second lemma involves the property of soft chunks.
Recall that for $d\in\N_0$, the soft chunks of $A$ are given by
$$B_d=\big(T_{1-\frac{1}{3d}}-T_{1-\frac{1}{2d}}\big)A.$$
\begin{lem}\label{p1o}
	Let $A\in M_2(\C)^{\otimes n}$ . Define
	$$G_{\mathrm{good}}=\Big\{d=2^l,l\in\N_0\  \big|\  dW_{\approx d}[A]\geq\frac{1}{20}\mathrm{Inf}[B_d]\Big\}.$$
	Then %there is an absolute constant $C>0$ such that
	$$\mathrm{Var}[A]\leq2\sum_{d\in G_{\mathrm{good}}}W_{\approx d}[A].$$
\end{lem}
\begin{proof}
	It is enough to show
	\begin{align}\label{e22}
		\sum_{d=2^l}\frac{\mathrm{Inf}[B_d]}{d}\leq 10\mathrm{Var}[A].
	\end{align}
	where the summation is over non-negative integer powers $d=2^l$ of $2$.
	Indeed, with (\ref{e22}) in hand and by the definition of variance $\mathrm{Var}[A]=\sum_{\substack{s\in \{0,1,2,3\}^n\\s\neq 0}}|\widehat A_s|^2$, one has
	\begin{align*}
		\sum_{d\in G_{\mathrm{good}}}W_{\approx d}[A]
		=  &\  \mathrm{Var}[A]-\sum_{d\notin G_{\mathrm{good}}}W_{\approx d}[A]\\
		\geq &\ \mathrm{Var}[A]-\frac{1}{20}\sum_{d\notin G_{\mathrm{good}}}\frac{\mathrm{Inf}[B_d]}{d}\\
		\geq&\ \mathrm{Var}[A]-\frac{1}{2}\mathrm{Var}[A]=\frac{1}{2}\mathrm{Var}[A].
	\end{align*}
	
	To prove (\ref{e22}), recall that $$\mathrm{Var}[A]=\sum_{s\neq 0}|\widehat A_s|^2\ \ \; \mbox{and}\ \ \; \widehat{(B_d)}_\emptyset= \tau(B_d) =  0.$$
	Hence, it suffices to show for each $s\neq0$, the terms $|\widehat A_s|^2$ appears in $\sum_d\frac{\mathrm{Inf}[B_d]}{d}$ with a multiplicative factor of at most 10. Observe that by (\ref{rwz25}) this factor is
	\begin{align}\label{e222}
		|\mathrm{supp} \, s|\sum_d\frac{1}{d}\Big(\Big(1-\frac{1}{3d}\Big)^{|\mathrm{supp} \, s|}-\Big(1-\frac{1}{2d}\Big)^{|\mathrm{supp} \, s|}\Big)^2.
	\end{align}
	We estimate the contribution from $d\leq |\mathrm{supp} \, s|$ and $d>|\mathrm{supp} \, s|$ separately.
	%Let $k$ be such that $2^k\leq |\mathrm{supp}s|<2^{k+1}$.
	Let $m$ be such that $2^m\leq|\mathrm{supp} \, s|<2^{m+1}$. The first part is bounded above by
	$$\sum_{d\leq|\mathrm{supp} \, s|}\frac{1}{d}\Big(1-\frac{1}{3d}\Big)^{2|\mathrm{supp} \, s|}\leq\sum_{d\leq|\mathrm{supp} \, s|}\frac{1}{d}e^{-\frac{2|\mathrm{supp} \, s|}{3d}}%\leq\sum_{j=1}^k\frac{1}{2^j}e^{-2^{k-j}}\leq2^{-k}\sum_{\ell=1}^\infty2^\ell e^{-2^\ell}
	\leq\sum_{i=0}^m\frac{1}{2^i}e^{-\frac{2}{3}2^{m-i}}\leq2^{-m}\sum_{\ell=0}^\infty2^{\ell}e^{-\frac{2}{3}2^{\ell}},%\leq
	%d 10
	$$
	%d  I CHANGED 10 to 2 HOPE THATS RIGHT
	which further bounded above by $\frac{5}{|\mathrm{supp} \, s|}$. On the other hand,
	the second part is bounded above by
	$$\sum_{d>|\mathrm{supp} \, s|}\frac{1}{d}\Big(1-\Big(1-\frac{1}{2d}\Big)^{|\mathrm{supp} \, s|}\Big)^2\leq\sum_{d>|\mathrm{supp} \, s|}\frac{1}{d}%\leq\sum_{j=1}^k\frac{1}{2^j}e^{-2^{k-j}}\leq2^{-k}\sum_{\ell=1}^\infty2^\ell e^{-2^\ell}
	\leq\frac{2}{|\mathrm{supp} \, s|}.$$
	We remind the reader that all the summation is over $d$ being powers of $2$. Finally, combining these two estimates, we see that the quality in (\ref{e222}) is bounded above by 10. Hence, the proof of (\ref{e22}) is completed.
\end{proof}
Now we are at a position to prove Theorem \ref{main3}.
\begin{proof}[Proof of Theorem \ref{main3}.]
	For any $d\in G_{\mathrm{good}}$, let $B_d$ be the soft chunks of $A$. Note that $\|B_d\|\leq2$. Then applying Lemma \ref{tala} to $B_d/2$, we conclude that one of the following holds for all $d$,
	\begin{itemize}
		\item[(i)]  $\displaystyle\sum_{j\in [n]} \, \frac{\mathrm{Inf}_j[B_d]}{\log({2}/\mathrm{Inf}_j^1[B_d])}\geq\frac{1}{432}\frac{d(W_{\approx d}[B_d])^2}{\mathrm{Inf}[B_d]}$.
		\item[(ii)] $\displaystyle\sum_{j\in [n]}\frac{\mathrm{Inf}_j^1[B_d]}{\log(2/\mathrm{Inf}_j^1[B_d])}\geq \frac{1}{384}dW_{\approx d}[B_d]$.
	\end{itemize}
	
	Assume that (i) holds.		
	By Lemma \ref{keylemma} (i) and (ii), we know that for each $j\in[n]$ and $d\in G_{\mathrm{good}}$, $\mathrm{Inf}_j^1[B_d]\leq2\mathrm{Inf}_j^1[A]$ and $\frac{1}{400}W_{\approx d}[A]\leq W_{\approx d}[B_d]\leq W_{\approx d}[A]$. On the other hand,  since $d\in G_{\mathrm{good}}$, $\mathrm{Inf}[B_d]\leq20 dW_{\approx d}[A]$. Therefore,
	$$\sum_{j\in[n]}\frac{\mathrm{Inf}_j[B_d]}{\log(1/\mathrm{Inf}_j^1[A])}\geq\sum_{j\in[n]}\frac{\mathrm{Inf}_j[B_d]}{\log(2/\mathrm{Inf}_j^1[B_d])}\geq\frac{1}{432}\frac{d(W_{\approx d}[B_d])^2}{\mathrm{Inf}[B_d]}\geq\frac{1}{10^{10}} W_{\approx d}[A].$$	
	%	$$W_{\approx d}[B_d]\approx W_{\approx d}[A]\ \mbox{and}\  \mathrm{Inf}_j^1(B_d)\lesssim\mathrm{Inf}_j^1(A).$$
	By above observation,  Lemma \ref{keylemma} (iii) and Lemma \ref{p1o}, we have
	\begin{align*}\sum_{j\in[n]}\frac{\mathrm{Inf}_j[A]}{\log(1/\mathrm{Inf}_j^1[A])}\geq &   \sum_{j\in[n]}\sum_{d\in G_{\mathrm{good}}}\frac{\mathrm{Inf}_j[B_d]}{\log(1/\mathrm{Inf}_j^1[A])}\\ \geq &\frac{1}{10^{10}}\sum_{d\in G_{\mathrm{good}}} W_{\approx d}[A]\geq\frac{1}{2\cdot10^{10}}\mathrm{Var}[A].\end{align*}
	Since  $\mathrm{Inf}_j[A]\leq2\mathrm{Inf}^1_j[A]$ for each $j\in[n]$, one has
	$$\sum_{j\in[n]}\frac{\mathrm{Inf}^1_j[A]}{\log(1/\mathrm{Inf}_j^1[A])}\geq\frac{1}{4\cdot10^{10}}\mathrm{Var}[A].$$
	This gives
	the desired assertion.
	%\begin{align*}
	%	\sum_{j\in[n]}\frac{\mathrm{Inf}_j[A]}{1+\log^+(\frac{1}{2^{2-p}\mathrm{Inf}_j^p[A]})}
	%	\geq    \sum_{j\in[n]}\sum_{d\in G_{\mathrm{good}}}\frac{\mathrm{Inf}_j[B_d]}{1+\log^+(\frac{1}{2^{2-p}\mathrm{Inf}_j^p[A]})}
	%	\geq  \frac{1}{10^9}\sum_{d\in G_{\mathrm{good}}} W_{\approx d}[A]\gtrsim\mathrm{Var}[A],
	%\end{align*}	
	
	If (ii) holds, using the facts $\mathrm{Inf}_j^1[B_d]\leq2\mathrm{Inf}_j^1[A]$ and $\frac{1}{400}W_{\approx d}[A]\leq W_{\approx d}[B_d]\leq W_{\approx d}[A]$ once more, one gets
	\begin{align*}\sum_{j\in [n]}\frac{2\mathrm{Inf}_j^1[A]}{\log(1/\mathrm{Inf}_j^1[A])}\geq &   \sum_{j\in[n]}\frac{\mathrm{Inf}^1_j[B_d]}{\log(2/\mathrm{Inf}_j^1[B_d])}\\ \geq &\frac{1}{384}dW_{\approx d}[B_d]\geq\frac{1}{153600}dW_{\approx d}[A].\end{align*}

	%\begin{align}\label{e2321222}
	%\sum_{j\in [n]}\frac{\mathrm{Inf}_j^1[A]}{1+\log^+(1/\mathrm{Inf}_j^1[A])}\gtrsim \sum_{j\in [n]}\frac{\mathrm{Inf}_j^1[B_d]}{1+\log^+(1/\mathrm{Inf}_j^1[B_d])}\gtrsim dW_{\approx d}[B_d]\gtrsim dW_{\approx d}[A].
	%\end{align}	
	
	On the other hand, by Lemma \ref{p1o}, one has
	\begin{align}\label{e23211222}
		\sum_{d\in G_{\mathrm{good}}}W_{\approx d}[A]\geq\frac{1}{2}\mathrm{Var}[A].
	\end{align}
	Then  \eqref{e23211222} implies that there exists $d\in G_{\mathrm{good}}$ such that
	\begin{align*}%\label{e2321112222}
		dW_{\approx d}[A]\geq \frac{1}{4}\mathrm{Var}[A],
	\end{align*}
	Otherwise, 	since $d=2^l, l\in \mathbb{N}_0$ and $\sum_{l=0}\frac{1}{2^l}=2$,
	\begin{align*}%\label{e2222}
		\sum_{d\in G_{\mathrm{good}}}W_{\approx d}[A]\le \sum_{d\in G_{\mathrm{good}}}\frac{1}{4d}\mathrm{Var}[A]\leq \frac{1}{2}\mathrm{Var}[A].
	\end{align*}
	which contradicts with (\ref{e23211222}). Therefore, by above observations we finally conclude that
	$$\mathrm{Var}[A]\leq1228800\sum_{j\in [n]}\frac{\mathrm{Inf}_j^1[A]}{\log(1/\mathrm{Inf}_j^1[A])}.$$ %\leq \frac{C_2}{(2-p)^2}\sum_{j\in[n]}\frac{\mathrm{Inf}^p_j[A]}{1+\log^+(1/\mathrm{Inf}_j^p[A])}.$$
	Hence, the proof is complete.
	%Clearly, the proof will be complete if we can show for any $d\in G_{\mathrm{good}}$ there exists $C>0$ such that
	%\begin{align}\label{e232222}
	%dW_{\approx d}[A]\geq C\mathrm{Var}[A].
	%\end{align}
	%	 Combining with Lemma \ref{tala}, we conclude that one of the following holds:
	%	 	\begin{itemize}
		%	 	\item[(i)]  $\sum_{j\in [n]}\frac{\mathrm{Inf}_j(B_d)}{\log(1/\mathrm{Inf}_j^1(A))}\gtrsim W_{\approx d}[A]$;
		%	 	\item[(ii)] $\sum_{j\in [n]}\frac{\mathrm{Inf}_j^1(A)}{\log(1/\mathrm{Inf}_j^1(A))}\gtrsim dW_{\approx d}[A]$.
		%	 \end{itemize}
	%To see this, suppose that (\ref{e232222}) does not hold, which implies that for
	%However,
	%
	%This implies that there is some $d\in D_{\mathrm{good}}$ such that $dW_{\approx d}[A]\gtrsim\mathrm{var}(A)$. If not,
	%
	%Therefore,
	% \begin{align*}
		% 	\sum_{j\in [n]}\frac{\mathrm{Inf}^1_j(A)}{\log(1/\mathrm{Inf}_j^1(A))}
		% 	\gtrsim\mathrm{var}(A).
		% \end{align*}	
	%So we finish the proof.
\end{proof}
\begin{remark}\label{quantumlp}
	Indeed, by using the same argument as in proving
	Theorem \ref{main3}, one can obtain the quantum $L^p$-Talagrand's inequality for $1\leq p<2$. More precisely, for any $A\in M_2(\C)^{\otimes n}$ with $\|A\|\leq1$ and any $1\leq p<2$, there exists a constant $C_p>0$ such that
		\begin{align}\label{e232111222}
	\mathrm{Var}[A]\leq C_p\sum_{j\in[n]}\frac{\mathrm{Inf}^p_j[A]}{\log(1/\mathrm{Inf}_j^p[A])},
	\end{align}
where the constant can be of order $C_p\approx\frac{C}{(2-p)^2}$ as $p\nearrow2$.
We note that such an inequality (\ref{e232111222}) has been obtained in \cite[Theorem 3.6]{RWZ} with better constant $C_p\approx\frac{C}{2-p}$. Nevertheless, in both cases the constant $C_p$ blows up as $p\nearrow2$, which leaves the orginial Talagrand's inequality with quantum $L^2$-influence still open.
\end{remark}

\section{Proof of Theorem \ref{thm:main2}}
In this section, %we show Theorem \ref{main4}.
we prove Theorem \ref{thm:main2} by using an interpolation method for semigroups. Let $J=\{j_1,\cdot\cdot\cdot,j_k\}
\subset [n]$ and $J^c$ be its complement. Recall that the $L^p$-influence of subset $J$ on $A$ is given by
$$\mathrm{Inf}_{J}^p[A]:=\|d_{j_1}\circ\cdot\cdot\cdot\circ d_{j_k}(A)\|_p^p.$$
%We first give the definition of the higher order quantum $L^p$-influence.
%\begin{defi}
%Let $A\in M_2(\C)^{\otimes n}$ and
%where $d_{i,j}=d_i\circ d_j$ and $d_i$ was defined in (\ref{eq:quantum_fourier_expansion1}).
%\end{defi}
For convenience, we write $d_{J}=d_{j_1}\circ\cdot\cdot\cdot\circ d_{j_k}$ and
\begin{align}\label{e2228}
d_J(A)=
\sum_{\substack{s\in \{0,1,2,3\}^n\\s_{j_1}\neq0,\cdot\cdot\cdot,s_{j_k}\neq0}}\widehat A_s\sigma_{s}=\sum_{\substack{s\in \{0,1,2,3\}^n\\ J\subset \mathrm{supp} \, s}}\widehat A_s\sigma_{s}.
\end{align}	
Note that the order does not matter here because $d_{i}\circ d_{j}=d_{j}\circ d_{i}$ for all $i,j\in[n]$.

%Let  $J=\{j_1,\cdot\cdot\cdot,j_k\}
%\subset [n]$ and $J^c$ be its complement.
Write $s_J=(s_{j_1},\cdot\cdot\cdot,s_{j_k})$ and $\sigma_{s_J}=\sigma_{s_{j_1}}\otimes\cdot\cdot\cdot\otimes\sigma_{s_{j_k}}$.
For $\gamma=(\gamma_1,\cdot\cdot\cdot,\gamma_k)\in\{1,2,3\}^k$, we  set $$\partial_J^\gamma (A):=\sum_{\substack{s\in \{0,1,2,3\}^n\\ s=(s_J,s_{J^c}),s_J=\gamma}}\widehat A_s\sigma_{s_{J^c}}.$$%\otimes \sigma_{s_{J}=0}.$$
Here %for $J=\{j_1,\cdot\cdot\cdot,j_k\}
%\subset [n]$,
$s_J=\gamma$ means $s_J=(s_{j_1},\cdot\cdot\cdot,s_{j_k})$ with  $s_{j_1}=\gamma_1,\cdot\cdot\cdot, s_{j_k}=\gamma_k$.  %and $s_J=0$ means $s_{j_1}=\cdot\cdot\cdot=s_{j_k}=0$.
%For $\gamma=(\gamma_1,\cdot\cdot\cdot,\gamma_k)\in\{1,2,3\}^k$, we  set $$\partial_J^\gamma (A):=\sum_{\substack{s\in \{0,1,2,3\}^n\\ s=(s_J,s_{J^c}),s_J=\gamma}}\widehat A_s\sigma_{s_{J^c}}.$$%\otimes \sigma_{s_{J}=0}.$$
%Here for $J=\{j_1,\cdot\cdot\cdot,j_k\}
%\subset [n]$,  $s_J=\gamma$ means $s_J=(s_{j_1},\cdot\cdot\cdot,s_{j_k})$ with  $s_{j_1}=\gamma_1,\cdot\cdot\cdot, s_{j_k}=\gamma_k$.  %and $s_J=0$ means $s_{j_1}=\cdot\cdot\cdot=s_{j_k}=0$.
\begin{lem}\label{integral}
	For any $A\in M_2(\C)^{\otimes n}$,  $J=\{j_1,\cdot\cdot\cdot,j_k\}
	\subset [n]$ and $\gamma=(\gamma_1,\cdot\cdot\cdot,\gamma_k)\in\{1,2,3\}^k$, we have	
	$$\partial_J^\gamma (A)=\tau_J\big(d_{J}(A)(1_{J^c}\ten \sigma_{s_{J}=\gamma})\big).$$
	%where $\sigma_{s_{J}=\gamma}=1_{J^c}\ten \sigma_{s_{J}}$ with $s_{J}=\gamma$.
	%\begin{itemize}
	%	\item[(i)]  $\partial_J^\gamma (A)=\tau_J\big(d_{J}(A)\sigma_{s_{J^c}=0}\otimes \sigma_{s_{J}=\gamma}\big)$; %As a consequence, $p_{j}\lesssim q_{j}\lesssim\sqrt{p_{j}}$.
	%	\item[(ii)] $\|g_{j,k}\|_2^2\leq\mathrm{Inf}_j[A]$; %As a consequence, $p_j\lesssim\mathrm{Inf}_j[A]$;
	%	\item [(iii)] $\|g_{j,k}\|_1\leq\mathrm{Inf}^1_j[A]$. %As a consequence, $\sum_{j\in J}q_j\leqm\mathrm{Inf}^1[A]$.
	%	\item [(ii)] $\|g_{j,k}\|_p^p\leq\mathrm{Inf}^p_j[A]$ for any $1\leq p\leq2$.
	%\end{itemize}	
\end{lem}
\begin{proof}
	By (\ref{e2228}), one has
	\begin{align*}
		\tau_J\big(d_{J}(A)(1_{J^c}\ten \sigma_{s_{J}=\gamma})\big)
		=&\  \tau_J\Big(\sum_{\substack{s\in \{0,1,2,3\}^n\\ J\subset \mathrm{supp} \, s}}\widehat A_s\sigma_{s_J}\sigma_{s_{J}=\gamma}\otimes\sigma_{s_{J^c}}\Big)\\
		=&\ \sum_{\substack{s\in \{0,1,2,3\}^n\\ s_{J}=\gamma}}\widehat A_s\sigma_{s_{J^c}}=\partial_J^\gamma (A),
	\end{align*}
	which finishes the proof.	
\end{proof}

\begin{lem}\label{se1}
	Let $1\leq k\leq n$ be integers.  For every $A\in M_2(\C)^{\otimes n}$, we have
	$$W_{\geq k}[A]=2k \, \sum_{J\subset [n], |J|=k}\sum_{\gamma\in\{1,2,3\}^k}\int_0^\infty(e^{2t}-1)^{k-1} \, e^{-2kt}\|P_t \, \partial_J^\gamma(A)\|_2^2 \, dt.$$
	%\begin{align*}
	%	&\ 2k \, \sum_{J\subset [n], |J|=k}\int_0^\infty(1-e^{-2t})^{k-1} \, \|P_t \, d_{j_1,\cdot\cdot\cdot,j_k}A\|_2^2 \, dt
	%	\\ =
	%	&\sum_{\substack{s\in \{0,1,2,3\}^n\\|\mathrm{supp} \,  s|\geq k}}\frac{(|\mathrm{supp} \, s|-1)\cdot\cdot\cdot(|\mathrm{supp} \, s|-k+1)}{(|\mathrm{supp} \, s|+k-1)\cdot\cdot\cdot(|\mathrm{supp} \, s|+1)} \; |\widehat A_s|^2	
%	\end{align*}	
	%	$$\sum_{\substack{s\in \{0,1,2,3\}^n\\|\mathrm{supp}s|\geq k}}\frac{(|\mathrm{supp}s|-1)\cdot\cdot\cdot(|\mathrm{supp}s|-d+1)}{(|\mathrm{supp}s|+d-1)\cdot\cdot\cdot(|\mathrm{supp}s|+1)}\widehat A_s^2=2d\sum_{1\leq i_1<i_2<\cdot\cdot\cdot<i_k\leq n}\int_0^\infty(1-e^{-2t})\|P_td_{i_1,i_2,\cdot\cdot\cdot,i_k}A\|_2^2dt.$$
\end{lem}
\begin{proof}
	Recall that for each $J\subset [n]$ with $|J|=k$ and each $\gamma\in\{1,2,3\}^k$,
	$$P_t \, \partial_J^\gamma(A)=\sum_{\substack{s\in \{0,1,2,3\}^n\\ s=(s_J,s_{J^c}),s_J=\gamma}}e^{-t(|\mathrm{supp} \, s|-k)}\widehat A_s\sigma_{s_{J^c}}.$$ %\sum_{J\subset \mathrm{supp}{s}}e^{-t|\mathrm{supp} \, s|} \, \widehat A_s \, \sigma_s.$$
	It follows from Parseval's identity, Fubini's theorem and Lemma \ref{integral} that
	\begin{align*}
		\sum_{|J|=k}\sum_{\gamma\in\{1,2,3\}^k}e^{-2kt}\|P_t \, \partial_J^\gamma(A)\|_2^2
		=&\ \sum_{|J|=k}\sum_{\gamma\in\{1,2,3\}^k}\sum_{\substack{s\in \{0,1,2,3\}^n\\ s=(s_J,s_{J^c}),s_J=\gamma}}e^{-2t|\mathrm{supp} \, s|} \; |\widehat A_s|^2\\
		=&\ \sum_{|\mathrm{supp} \, s|\geq k}\sum_{\substack{J\subset \mathrm{supp}{s} \\ |J|=k }}e^{-2t|\mathrm{supp} \, s|} \; |\widehat A_s|^2\\
		=&\ \sum_{|\mathrm{supp} \, s|\geq k}{\,|\mathrm{supp} \, s|\, \choose k} \, e^{-2t|\mathrm{supp} \, s|} \; |\widehat A_s|^2.
	\end{align*}	
	Therefore, we find
	\begin{align*}
		&\
		\sum_{J\subset [n], |J|=k}\sum_{\gamma\in\{1,2,3\}^k}\int_0^\infty(e^{2t}-1)^{k-1} \, e^{-2kt}\|P_t \, \partial_J^\gamma(A)\|_2^2 \, dt\\
		= &\ \sum_{|\mathrm{supp} \, s|\geq k}{\,|\mathrm{supp} \, s|\, \choose k} \, |\widehat A_s|^2\int_0^\infty(e^{2t}-1)^{k-1}e^{-2t|\mathrm{supp} \, s|} \, dt.%\\
		%	= &\ \frac{1}{4}\sum_{\substack{s\in \{0,1,2,3\}^n\\|\mathrm{supp}s|\geq2}}\frac{|\mathrm{supp}s|-1}{|\mathrm{supp}s|+1}\widehat A_s^2,
	\end{align*}	
	Substituting $u = e^{-2t}$, $u \in (0,1]$, i.e.\ $t = -\frac{1}{2}\log u$ with $\big|\frac{dt}{du}\big| = \frac{1}{2u}$, we get
	\begin{align*}
		\int_0^\infty(e^{2t}-1)^{k-1}e^{-2t|\mathrm{supp} \, s  |} \, dt
		=&\  \frac{1}{2}\int_0^1(1-u)^{k-1}u^{|\mathrm{supp} \, s|-k} \, du
		%	=&\ \frac{1}{2}B(k,|\mathrm{supp} \, s|)
		\\ =&\frac{1}{2}B(k,|\mathrm{supp} \, s|-k+1)
		=\frac{1}{2k}{\,|\mathrm{supp} \, s|\, \choose k}^{-1},
	\end{align*}	
	which yields the statement.
\end{proof}

We start with the proof of (\ref{eq:ht}) stated in Theorem \ref{thm:main2}.
\begin{proof}[Proof of (\ref{eq:ht}).]
 %We assume that $k\geq2$ since  the case $k=1$ is exactly Theorem \ref{main3} (see also \cite[Theorem 3.6]{RWZ}). %\LG{maybe cite the theorem in haonan's work.}	
%Let $A\in M_2(\C)^{\otimes n}$ with $\|A\|\leq1$.
%For each $2\leq k\leq n$, it follows from Lemma \ref{tals} that
%$$W_{\geq k}[A]\leq\frac{2(k+1)!}{k}\sum_{\substack{s\in \{0,1,2,3\}^n\\|\mathrm{supp} \, s|\geq k}}\frac{(|\mathrm{supp} \, s|-1)\cdot\cdot\cdot(|\mathrm{supp} \, s|-k+1)}{(|\mathrm{supp} \, s|+k-1)\cdot\cdot\cdot(|\mathrm{supp} \, s|+1)} \; |\widehat A_s|^2.$$
It is not difficult to conclude from the proof of Lemma \ref{was21} that for each subset $J$,  $\mathrm{Inf}_{J}^1[A]\leq1$. If $\mathrm{Inf}_{J}^1[A]=1$, then the right-hand side of  (\ref{eq:ht}) equals $+\infty$. Hence, it suffices to consider $\mathrm{Inf}_{J}^1[A]<1$.

We claim that it is enough to show for any subset $J$ with $|J|=k$ and any $\gamma\in\{1,2,3\}^k$, %there exists an absolute constant $C>0$ such that
\begin{equation}\label{8rt}
	\int_0^\infty(e^{2t}-1)^{k-1} e^{-2kt} \,  \|P_t \, \partial_J^\gamma(A)\|_2^2 \,  dt\leq  \frac{1}{2}8^k (k-1)!\, \frac{\mathrm{Inf}_{J}^1[A]%(1+\mathrm{Inf}_{i,j}^1(A))
	}{[\log(1/\mathrm{Inf}_{J}^1  [A])]^k}.
\end{equation}
Indeed, combining with  Lemma \ref{se1}, we get
\begin{align*}
	W_{\geq k}[A]
	\leq&\ 8^kk! \, \sum_{J\subset [n], |J|=k}\sum_{\gamma\in\{1,2,3\}^k}\frac{\mathrm{Inf}_{J}^1[A]
	}{[\log(1/\mathrm{Inf}_{J}^1  [A])]^k}\\
	=&\ 24^kk! \, \sum_{J\subset [n], |J|=k}\frac{\mathrm{Inf}_{J}^1[A]
	}{[\log(1/\mathrm{Inf}_{J}^1  [A])]^k},%\\
	%=&\ \|d_{i,j}A\|_1\Big(\frac{1}{\log\big(1\setminus\|d_{i,j}A\|_1\big)}\Big).
\end{align*}
%Then by changing variables to $u = t/2$, one has
%\begin{align*}
%	\int_0^\infty(1-e^{-2t})^{k-1}  \, \|P_t \, d_{J}A\|_2^2 \; dt
%	= &\ 2 \int_0^\infty (1-e^{-4u})^{k-1} \, \|P_{2u} \, d_{J}A\|_2^2 \, du\\
%	\leq&\ 2^k\int_0^\infty (1-e^{-2u})^{k-1}e^{-2u} \, \|P_{u} \, d_{J}A\|_2^2 \, du.
%In the last inequality we used the Poincar\'e inequality- Lemma \ref{basic} i)
%and the basic inequality $1-e^{-4u}\leq2(1-e^{-2u})$ for $u\geq0$. Furthermore, for any $u\geq0$,
as desired.
To show (\ref{8rt}),
we use  hypercontractivity-Lemma \ref{basic} ii) to get
$$\|P_t \, \partial_J^\gamma(A)\|_2%=\|P_{s}d_{i,j}P_sA\|_2
\leq\|\partial_J^\gamma(A)\|_{p(t)}$$
with $p(t)=1+e^{-2t}$.
Note that by Lemma \ref{integral} and H\"older's inequality,
$$\|\partial_J^\gamma\|_{\infty\rightarrow\infty}\leq\|d_{J}\|_{\infty\rightarrow\infty}\leq2^k;$$
and for each $t\geq0$, %\textcolor{red}{Here, we can have $2^{k(1-\frac{1}{p(t)})}$. Will it improve the estimate?}
$$\|\partial_J^\gamma(A)\|_{p(t)}\leq\|\partial_J^\gamma(A)\|_1^{1/p(t)}\|\partial_J^\gamma(A)\|_\infty^{1-1/p(t)}\leq2^k\|\partial_J^\gamma(A)\|_1^{1/p(t)}\leq2^k\mathrm{Inf}_{J}^1  [A]^{1/p(t)}.$$
Altogether, we deduce that
\begin{align*}
	&\
	\int_0^\infty(e^{2t}-1)^{k-1} e^{-2kt} \,  \|P_t \, \partial_J^\gamma(A)\|_2^2 \,  dt\\
	\leq&\  4^k \, 	\int_0^\infty (e^{2t}-1)^{k-1} \, e^{-2kt} \, \mathrm{Inf}_{J}^1  [A]^{2/p(t)} \, dt\\
	=&\ 4^k \, \mathrm{Inf}_{J}^1  [A]	\int_0^\infty (e^{2t}-1)^{k-1} \, e^{-2kt} \, \mathrm{Inf}_{J}^1  [A]^{2/p(t)-1}  \, dt.
\end{align*}	
	By letting  $s=\frac{1-e^{-2t}}{1+e^{-2t}} \in [0,1]$,  one has $t=\frac{\log(1+s)-\log(1-s)}{2}$ and
	$\frac{d t}{d s}=\frac{1}{2}\left(\frac{1}{1+s}+\frac{1}{1-s}\right)=\frac{1}{(1-s)(1+s)}$. Then
	%\begin{align*}
	%	\int_0^\infty (e^{2t}-1)^{k-1} \, e^{-2kt}  dt
	%	=&\ \int_0^1 \Big(\frac{2s}{1-s}\Big)^{k-1} \Big(\frac{1-s}{1+s}\Big)^{k}\frac{1}{(1+s)(1-s)} \, ds\\
	%	\leq&\ 2^{k-1}\int_0^1 s^{k-1} ds\leq2^{k-1}.%\\
	%=&\ \|d_{i,j}A\|_1\Big(\frac{1}{\log\big(1\setminus\|d_{i,j}A\|_1\big)}\Big).
	%\end{align*}
	%Therefore,
	%$$\int_0^\infty(e^{2t}-1)^{k-1} e^{-2kt} \,  \|P_t \, \partial_J^\gamma(A)\|_2^2 \,  dt\leq16^k\mathrm{Inf}_{J}^1  [A],$$
	%which gives (\ref{8rt}).
	%
	% If $\mathrm{Inf}_{J}^1  [A]<1$, by changing variable $s=\frac{1-e^{-2t}}{1+e^{-2t}}$ again, we obtain
	%since for any $s\geq0$
	%$$\frac{2-p(s)}{p(s)}=\frac{e^{2s}-1}{e^{2s}+1}$$
	%set $a=\|d_{i,j}A\|_1$ and
	%by a change of variables, we obtain
	\begin{align*}
		&\
		\int_0^\infty (e^{2t}-1)^{k-1} \, e^{-2kt} \, \mathrm{Inf}_{J}^1  [A]^{2/p(t)-1}  \, dt\\
		=&\ \int_0^1 \Big(\frac{2s}{1-s}\Big)^{k-1} \Big(\frac{1-s}{1+s}\Big)^{k}\frac{1}{(1+s)(1-s)}\, \mathrm{Inf}_{J}^1  [A]^s \, ds\\
		\leq&\ 2^{k-1}\int_0^1 s^{k-1} \, \mathrm{Inf}_{J}^1  [A]^s \, ds. %\\
		%=&\ \|d_{i,j}A\|_1\Big(\frac{1}{\log\big(1\setminus\|d_{i,j}A\|_1\big)}\Big).
	\end{align*}
	%Note that
	%\[ \norm{d_J(A)}{1}\le \norm{d_J(A)}{2}=\norm{(I-E_J)(A)}{2}\le\norm{A}{2}\le \norm{A}{}=1 \]	
	Since for any $0<a<1$,
	$$\int_0^1 s^{k-1} \, a^s \, ds=\int_0^1 s^{k-1} \, e^{-s\log\frac{1}{a}} \, ds=\frac{1}{[\log1/a]^k}\int_0^{\log1/a} y^{k-1}e^{-y}
	\, dy\le \frac{\Gamma(k)}{[\log1/a]^k},$$
	where $\Gamma(k)=\int_0^\infty s^{k-1} \, e^{-s} \, ds$ is the Gamma function.
	As a consequence, we finally get
	\begin{align*}
		\int_0^\infty (e^{2t}-1)^{k-1} \, e^{-2kt} \, \mathrm{Inf}_{J}^1  [A]^{2/p(t)-1}  \, dt
		\le&\  2^{k-1} (k-1)!\, \frac{\mathrm{Inf}_{J}^1  [A]}{[\log1/\mathrm{Inf}_{J}^1  [A]]^k},
	\end{align*}	
	which yields (\ref{8rt}). Hence, the proof is complete.
\end{proof}

\begin{proof}[Proof of (\ref{eq:ht1}).]
	We have already proved the case $k=1$ (see Remark \ref{quantumkkl}). Assume that $k\geq2$. We consider two cases. If $$ \max_{|J|=k}\mathrm{Inf}_{J}^1[A]\geq\frac{1}{n^{k/2}},$$ then since $\log n\leq\sqrt{n}$ for all $n\in\N$, one has
$$\max_{|J|=k}\mathrm{Inf}_{ J}^1[A]\geq\Big(\frac{\log n%(1+\mathrm{Inf}_{i,j}^1(A))
}{n}\Big)^k\geq W_{\geq k}[A]\Big(\frac{\log n%(1+\mathrm{Inf}_{i,j}^1(A))
}{n}\Big)^k,$$	
as desired.

If for all subset $J$ with $|J|=k$,	$\mathrm{Inf}_{J}^1[A]\leq\frac{1}{n^{k/2}}$, then by (\ref{eq:ht}), %\LG{Check the new constant below}
\begin{align*}
	W_{\geq k}[A]
	\leq &\ 24^k \, k!\sum_{|J|=k}\frac{\mathrm{Inf}_{J}^1[A]%(1+\mathrm{Inf}_{i,j}^1(A))
	}{\log(1/\mathrm{Inf}_{J}^1[A])^k}\\
	\leq &\ 24^k \, k!\Big(\frac{1}{\frac{k}{2}\log n}\Big)^k\sum_{|J|=k}\mathrm{Inf}_{J}^1[A]
	\\
	\leq &\ 24^k \, k!\Big(\frac{1}{\frac{k}{2}\log n}\Big)^k\binom{n}{k}\max_{|J|=k}\mathrm{Inf}_{J}^1[A]\\
	%\leq &\ C16^k \, (k+2)!\Big(\frac{1}{\frac{k}{2}\log n}\Big)^k{\,n\, \choose k}\max_{1\leq j_1<\cdot\cdot\cdot<j_k\leq n}\mathrm{Inf}_{j_1,\cdot\cdot\cdot,j_k}^1[A]\\
	\leq &\ \frac{48^k}{k^k}\Big(\frac{n}{\log n}\Big)^k\max_{|J|=k}\mathrm{Inf}_{J}^1[A]\\
	\leq &\ C\Big(\frac{n}{\log n}\Big)^k\max_{|J|=k}\mathrm{Inf}_{J}^1[A],
\end{align*}
where $C$ is some absolute constant. This finishes the proof.
\end{proof}

%We now turn to quantum $L^1$-Talagrand's inequality at order $k$ with $k\geq3$. For any $k\geq3$ and any $(j_1,\cdot\cdot\cdot,j_k)\in[n]^k$, the $L^1$-influence of $(j_1,\cdot\cdot\cdot,j_k)$ is given by
%$$\mathrm{Inf}_{j_1,\cdot\cdot\cdot,j_k}^1[A]=\|d_{j_1,\cdot\cdot\cdot,j_k}(A)\|_1,$$
%where $d_{j_1,\cdot\cdot\cdot,j_k}=d_{j_1}\circ\cdot\cdot\cdot \circ d_{j_k}$. It can be eaily seen that the previous argument also works for quantum $L^1$-Talagrand' inequality at order $k$ by noting $\mathrm{tr}(d_{j_1,\cdot\cdot\cdot,j_k}A)=0$ for all $(j_1,\cdot\cdot\cdot,j_k)\in[n]^k$. We  leave the details to the reader.
%\begin{thm}\label{main9}
%	For any $n\geq1$, any $k\geq3$ and any $A\in M_2(\C)^{\otimes n}$ with $\|A\|\leq1$, there exists an absolute constant $C>0$ such that $$\mathrm{Var}[A]\leq C\sum_{j_1,\cdot\cdot\cdot,j_k\in [n]}\frac{\mathrm{Inf}_{j_1,\cdot\cdot\cdot,j_k}^1[A]%(1+\mathrm{Inf}_{i,j}^1(A))
%	}{[1+\log^+(1/\mathrm{Inf}_{j_1,\cdot\cdot\cdot,j_k}^1[A])]^k}.$$
%	where $\mathrm{Inf}_{i,j}^1[A]:=\|d_i\circ d_j(A)\|_1$.
%\end{thm}
%Theorem \ref{main9} and the argument in proving Theorem \ref{main5} allow us to obtain the quantum $L^1$-KKL theorem at order $k$.
%\begin{cor}\label{main25}
%	For any $n\geq1$, any $k\geq3$ and any quantum boolean function $A\in M_2(\C)^{\otimes n}$ with $\mathrm{tr}(A)=0$, there exists an absolute constant $C>0$ such that $$\max_{1\leq j_1,\cdot\cdot\cdot,j_k\leq n}\mathrm{Inf}_{j_1,\cdot\cdot\cdot,j_k}^1[A]\geq C\Big(\frac{\log n%(1+\mathrm{Inf}_{i,j}^1(A))
%	}{n}\Big)^k.$$
%\end{cor}

\section{Proof of Theorem \ref{thm:main3}}
The proof of Theorem \ref{thm:main3} is based on the following level-inequality, which is of independent interest.
\begin{lem}\label{l8}
	Let $d\geq2$ and  $A\in M_2(\C)^{\otimes n}$ with $\|A\|\leq1$ such that
	\begin{align}\label{e22122222}
		\sum_{j=1}^n\mathrm{Inf}_{j}^1[A]^2\leq\exp(-2(d-1)).
	\end{align}
	Then there exists an absolute constant $C>0$ such that
	$$\sum_{|\mathrm{supp} \, \nu|=d}|\widehat A_{\nu}|^2\leq\frac{Ce}{d}\Big(\frac{2e}{d-1}\Big)^{d-1}\sum_{j=1}^n\mathrm{Inf}_{j}^1[A]^2\Big(\log\Big(\frac{d}{\sum\limits_{j=1}^n\mathrm{Inf}_{j}^1[A]^2}\Big)\Big)^{d-1},$$
	where $\widehat A_\nu$ is the Fourier coefficient of $A$.
\end{lem}
Lemma \ref{l8} is a generalization of a lemma by Talagrand \cite{T2} (see also \cite{T3}). Talagrand's lemma gives a bound on the second-level Fourier coefficients for a monotone Boolean function in terms of its weight on first level coefficients. Later on,  Keller and Kindler \cite{KK} simplified the proof of Talagrand's lemma and give a  similar bound on the
weight on $d$-level coefficients for general Boolean functions.

We start with Lemma \ref{l8}, whose proof depends on the following two lemmas. %Throughout the proof, we denote
%$$\mathcal W[A]=\sum_{j=1}^n\mathrm{Inf}_{j}^1[A]^2.$$
%The first lemma gives a bound of $\mathrm{tr}(\chi_{\{|A|\geq t\}})$ for any  $A\in M_2(\C)^{\otimes n}$ whose Fourier degree is at most $d$ with $d\geq2$.
%The first is the Hypercontractive obtained  in \cite{MO}.
%\begin{lem}\label{hyper}
%	For any $A\in M_2(\C)^{\otimes n}$ of degree at most $d$ and for any $p\geq2$,
%	$$\|A\|_p\leq(p-1)^{d/2}\|A\|_2.$$
%\end{lem}
%\begin{lem}\label{l1}
%	Let $A\in M_2(\C)^{\otimes n}$ be of degree at most  $d$ with $\|A\|_2=1$. Then for any $t\geq(2e)^{d/2}$,
%	$$\frac{1}{2^n}\mathrm{tr}(\chi_{\{|A|\geq t\}})\leq\exp\big(-\frac{d}{2e}t^{2/d}\big).$$
%\end{lem}
%\begin{proof}
%	The key to the proof is to use the following fact, which is a consequence of hypercontractive-Lemma \ref{noise1} (see e.g. \cite[Corollary 8.9]{MO}):  	for any $A\in M_2(\C)^{\otimes n}$ with of degree at most $d$ and for any $p\geq2$,
%	\begin{align}\label{e21222}
%	\|A\|_p\leq(p-1)^{d/2}\|A\|_2.
%	\end{align}
%	Set $p=(t^{2/d}/e)$. Then by assumption of $t$, $p\geq2$. Hence, combining with Markov's inequality and noting $\|A\|_2=1$, one has
%	$$\frac{1}{2^n}\mathrm{tr}(\chi_{\{|A|\geq t\}})\leq\frac{\|A\|_p^p}{t^p}\leq\frac{\big((p-1)^{d/2}\|A\|_2\big)^p}{t^p}\leq\Big(\frac{p^{d/2}}{t}\Big)^p=\exp\big(-\frac{d}{2e}t^{2/d}\big),$$	
%	as desired.
%\end{proof}
The first lemma can be found in \cite[Lemma 12]{KK}. The statement there contains a term $B(p)$. Below is a special of $B(p) = 1$ by choosing $p=\frac{1}{2}$. We refer to \cite{KK} for its proof.
\begin{lem}\label{l2}
	Let $d\geq2$ be an integer and let $t_0$ be such that $t_0>( 4e)^{(d-1)/2}$. Then
	$$\int_{t_0}^\infty t^2\exp\Big(-\frac{d-1}{2e}\cdot t^{2/(d-1)}\Big)dt\leq5e \, t_0^{3-\frac{2}{d-1}} \, \exp\Big(-\frac{d-1}{2e}\cdot t_0^{2/(d-1)}\Big).$$
\end{lem}
With the above
% lemmas in hand,
we shall prove the lemma below.
Recall that for $k\in\{1,2,3\}$, $$\sigma_{k(j)}=
\sigma_0^{\otimes j-1}\otimes\sigma_{k}\otimes\sigma_0^{n-j}\ .$$
For $J\subset [n]$, we denote
\[\mathcal V_J[A]=\sum_{j\in J}\mathrm{Inf}_{j}^1[A]^2.\]
%\sigma_{\{0,\cdot\cdot\cdot0,s_j,0,\cdot\cdot,0\}}$ %, $k\in\{1,2,3\}$.  %and  Lemma \ref{l8} can be obtained by taking expectation over a suitable random subset.
\begin{lem}\label{l3}
	%Let $J\in[n]$ and $J^c=[n]\setminus J$.
	Let $J\subset [n]$ be subset. Let $A\in M_2(\C)^{\otimes n}$ with $\|A\|\leq1$ and $d\geq2$. Assume that
\begin{align}\label{e221222222}
		\sum_{j=1}^n\mathrm{Inf}_{j}^1[A]^2\leq\exp(-2(d-1)).
	\end{align}
	 Then we have
	\begin{align}\label{e22221222}
\sum_{\substack{v\in\{0,1,2,3\}^{n}\\|\mathrm{supp }\, v\cap J|=1\\ |\mathrm{supp }\, v|=d}}|\widehat A_{v}|^2
		=&\sum\limits_{\substack{s\in\{0,1,2,3\} ^{J^c}\\|\mathrm{supp} \, s|=d-1}}\sum\limits_{\substack{t\in \{0,1,2,3\}^{J}\\|\mathrm{supp} \, t|=1}}|\widehat A_{s\cup t}|^2\\ \leq &10240\Big(\frac{2e}{d-1}\Big)^{d-1}\mathcal V_J[A]\Big(\log\Big(\frac{1}{\mathcal V_J[A]}\Big)\Big)^{d-1}.
	\end{align}
\end{lem}
\begin{proof}For $j\in J$ and $k\in\{1,2,3\}$, define
	$$h^{\prime}_{j,k}=\sum\limits_{\substack{s\in\{0,1,2,3\} ^{J^c}\\|\mathrm{supp} \, s|=d-1}}\sum\limits_{\substack{t\in \{0,1,2,3\}^{J}\\ \mathrm{supp} \, t=\{j\},t_j=k}}\widehat A_{s\cup t}\sigma_s.$$
Note that $\mathrm{supp} \, h^{\prime}_{j,k}$ is contained in $J^c$. Then
	\begin{align*}\langle h^{\prime}_{j,k}\sigma_{k(j)},A\rangle=&\Big\langle \sum\limits_{\substack{s\in\{0,1,2,3\} ^{J^c}\\|\mathrm{supp}  \, s|=d-1}}\sum\limits_{\substack{t\in \{0,1,2,3\}^{J}\\|\mathrm{supp} \, t=\{j\},t_j=k}}\widehat A_{s\cup t}\sigma_{s\cup t},A\Big\rangle\\ =&\sum\limits_{\substack{s\in\{0,1,2,3\} ^{J^c}\\|\mathrm{supp} \, s|=d-1}}\sum\limits_{\substack{t\in \{0,1,2,3\}^{J}\\\mathrm{supp} \, t=\{j\},t_j=k}}|\widehat A_{s\cup t}|^2 \\ =&\norm{h^{\prime}_{j,k}}{2}^2
\end{align*}
	Summing over $j\in J$ and $k\in\{1,2,3\}$ we obtain	
	\begin{align*}
		\big\langle \sum_{k=1}^3\sum_{j\in J}h^{\prime}_{j,k}\sigma_{k(j)},A\big\rangle	= &\ \sum_{k=1}^3\sum_{j\in J} \sum\limits_{\substack{s\in\{0,1,2,3\} ^{J^c}\\|\mathrm{supp}  \, s|=d-1}}\sum\limits_{\substack{t\in \{0,1,2,3\}^{J}\\\mathrm{supp} \, t={j},t_j=k}}|\widehat A_{s\cup t}|^2=\sum_{\substack{v\in\{0,1,2,3\}^{n}\\|\mathrm{supp }\, v\cap J|=1\\ |\mathrm{supp }\, v|=d}}|\widehat A_{v}|^2, %\\
	%	= &\ \sum\limits_{\substack{s\in\{0,1,2,3\} ^{J^c}\\|\mathrm{supp}s|=d-1}}\sum\limits_{\substack{t\in \{0,1,2,3\}^{J}\\|\mathrm{supp}t|=1}}|\widehat A_{s\cup t}|^2.
	\end{align*}
which is exactly the left-hand side of (\ref{e22221222}).
	%$$\big\langle \sum_{j\in J}\sum_{k=1}^3h^{\prime}_j\sigma_{k(j)},A\big\rangle=\sum\limits_{\substack{s\in\{0,1,2,3\} ^{J^c}\\|\mathrm{supp}s|=d-1}}\sum_{j\in J}\sum_{k=1}^3\sum\limits_{\substack{t\in \{0,1,2,3\}^{J}\\|\mathrm{supp}t|=1,t_j=k}}|\widehat A_{s\cup t}|^2=\sum\limits_{\substack{s\in\{0,1,2,3\} ^{J^c}\\|\mathrm{supp}s|=d-1}}\sum\limits_{\substack{t\in \{0,1,2,3\}^{J}\\|\mathrm{supp}t|=1}}|\widehat A_{s\cup t}|^2.$$	
	For convenience, let us normalize $h^{\prime}_{j,k}$ by taking  $h_{j,k}=h^{\prime}_{j,k}/\|h^{\prime}_{j,k}\|_2$. Then
	\begin{align*}%\label{e2222222}
		\langle h_{j,k}\sigma_{k(j)},A\rangle= \| h^{\prime}_{j,k} \|_2.
	\end{align*}
By the observations above  it suffices to show
	\begin{align}\label{e2222222}
		\sum_{k=1}^3\sum_{j\in J}(	\langle h_{j,k}\sigma_{k(j)},A\rangle)^2\leq10240\Big(\frac{2e}{d-1}\Big)^{d-1}\mathcal V_J[A]\Big(\log\Big(\frac{1}{\mathcal V_J[A]}\Big)\Big)^{d-1}.
	\end{align}
	
	Let us now focus on the proof of (\ref{e2222222}). Since $h_{j,k}$ and $\sigma_{k(j)}$ commute for each $j\in J$ and $k\in\{1,2,3\}$, one has
	\begin{align*}
		%\sum\limits_{\substack{s\in\{0,1,2,3\} ^{J^c}\\|\mathrm{supp}s|=d-1}}|\widehat A_{\{j\}\cup s}|^2
		\big(\langle h_{j,k}\sigma_{k(j)},A\rangle\big)^2=\big(\langle h_{j,k},\sigma_{k(j)}A\rangle\big)^2
		=  \Big(\tau_{J^c}\Big[h_{j,k}\cdot\tau_J[\sigma_{k(j)}A]\Big]\Big)^2.%\\ \leq &\  \textcolor{red}{\Big(\tau_{J^c}\Big[|h_j|\cdot\big|\tau_J[\sigma_jA]\big|\Big]\Big)^2}\\
		%	=&\ \Big(\int_0^\infty\tau_{J^c}\Big[1_{|h_j|>t}\cdot\big|\tau_J[\sigma_jA]\big|\Big]dt\Big)^2.
	\end{align*}
	Recall that $g_{j,k}=\tau_J[\sigma_{k(j)}A]$. %=b_{j,+}-b_{j,-}$.
	Decompose \begin{align*}&h_{j,k}=h_{j,k,1}-h_{j,k,2}+i(h_{j,k,3}-h_{j,k,4})\\ &g_{j,k}=g_{j,k,1}-g_{j,k,2}+i(g_{j,k,3}-g_{j,k,4})\ ,\end{align*} where $h_{j,k,\ell}$ (resp.\ $g_{j,k,\ell}$), $\ell=1,2,3,4$ are positive operators with $\|h_{j,k,\ell}\|_q\leq\|h_{j,k}\|_q$ (resp. $\|g_{j,k,\ell}\|_q\leq\|g_{j,k}\|_q$) for all $q\geq1$ and $\ell=1,2,3,4$. As a consequence,  %$(	\langle h_j\sigma_{k(j)},A\rangle)^2$ is bounded above by
	%	Note that $h_j=h_{j,+}-h_{j,-}$ and $b_j:=\tau_J[\sigma_jA]=b_{j,+}-b_{j,-}$. Hence, the above is bounded by
	\begin{align}\label{e2221212222}
		(	\langle h_{j,k}\sigma_{k(j)},A\rangle)^2\leq16\sum_{\ell=1}^4\sum_{\rho=1}^4\Big(\tau_{J^c}[h_{j,k,\ell}g_{j,k,\rho}]\Big)^2.
	\end{align}
	%\Big[\Big(\tau_{J^c}[h^+_{j}b^+_{j}]\Big)^2+\Big(\tau_{J^c}[h^+_{j}b^-_{j}]\Big)^2+\Big(\tau_{J^c}[h^-_{j}b^+_{j}]\Big)^2+\Big(\tau_{J^c}[h^-_{j}b^-_{j}]\Big)^2\Big].$$
	
	Consider  the term $\Big(\tau_{J^c}[h_{j,k,1}g_{j,k,1}]\Big)^2$ firstly. We will show
	\begin{align}\label{e222122222}
		\sum_{k=1}^3\sum_{j\in J}	\Big(\tau_{J^c}[h_{j,k,1}g_{j,k,1}]\Big)^2\leq40\Big(\frac{2e}{d-1}\Big)^{d-1}\mathcal V_J[A]\Big(\log\Big(\frac{1}{\mathcal V_J[A]}\Big)\Big)^{d-1}.
	\end{align}
	To see this, by Fubini's theorem and Cauchy-Schwarz inequality we see that for any parameter $t_0>0$
	%In the following, we only deal with the first term above, while the arguments for the other three terms are exactly the same. Observe that
	\begin{align*}
		\Big(\tau_{J^c}[h_{j,k,1}g_{j,k,1}]\Big)^2	= &\ \Big(\int_0^\infty\tau_{J^c}\Big[\chi_{\{h_{j,k,1}>t\}}g_{j,k,1}\Big]dt\Big)^2\\
		\leq &\ 2\Big(\int_0^{t_0}\tau_{J^c}\Big[\chi_{\{h_{j,k,1}>t\}}g_{j,k,1}\Big]dt\Big)^2\\
		&\ +2\Big(\int_{t_0}^\infty\tau_{J^c}\Big[\chi_{\{h_{j,k,1}>t\}}g_{j,k,1}\Big]dt\Big)^2:=I_{j,k}+II_{j,k}
	\end{align*}
where $\chi_{\{h_{j,k,1}>t\}}$ is the support projection of the positive part of $h_{j,k,1}-t{\bf 1} $.
	%	\begin{align*}
		%	&\quad	\Big(\tau_{J^c}[h^+_{j}b^+_{j}]\Big)^2\leq2\Big(\int_0^{t_0}\tau_{J^c}\Big[1_{h_{j,+}>t}b_{j,+}\Big]dt\Big)^2\\
		%	& \ \ \ \ \ \ \ +2\Big(\int_{t_0}^\infty\tau_{J^c}\Big[1_{h_{j,+}>t}b_{j,+}\Big]dt\Big)^2:=I+II.
		%\end{align*}
		%$$\Big(\tau_{J^c}[h_{j,+}b_{j,+}]\Big)^2= \Big(\int_0^\infty\tau_{J^c}\Big[1_{h_{j,+}>t}b_{j,+}\Big]dt\Big)^2.$$
		%	
%		for any $t_0>0$, which will be determined later. %Then by Cauchy-Schwarz inequality,
%		
		%\textbf{Bounding $I$}
		The first term is easy to handle. Indeed, by using the fact $\|g_{j,k,1}\|_1\leq\|g_{j,k}\|_1$ and Lemma \ref{restrict}, we have% H\"older's inequality,
		\begin{align*}
			I_{j,k}	\leq &\ 2\Big(\int_0^{t_0}\tau_{J^c}[g_{j,k,1}]dt\Big)^2\leq2\Big(\int_0^{t_0}\tau_{J^c}\Big[\Big|\tau_J[\sigma_{k(j)}A]\Big|\Big]dt\Big)^2\\
			= &\ 2t_0^2 \, \Big(\tau_{J^c}\Big[\Big|\tau_J[d_j(A) \, \sigma_{k(j)}]\Big|\Big]\Big)^2\leq 2t_0^2 \, \mathrm{Inf}_{j}^1[A]^2.
		\end{align*}
		%ADCH above
For the second term, we apply the Cauchy-Schwarz inequality twice to get
		\begin{align*}
			II_{j,k}	\leq &\ 2\Big(\int_{t_0}^\infty \frac{1}{t^2}dt\Big)\int_{t_0}^\infty t^2\Big(\tau_{J^c}\big[\chi_{\{h_{j,k,1}>t\}}g_{j,k,1}\big]\Big)^2dt\\
			\leq &\ \frac{2}{t_0}\int_{t_0}^\infty t^2\tau_{J^c}[\chi_{\{h_{j,k,1}>t\}}] \, \tau_{J^c}\big[|g_{j,k,1}|^2\big] \, dt.
		\end{align*}
	Assume that $t_0>(4e)^{(d-1)/2}$. Then for any $t\geq t_0$, Markov's inequality
	gives
	$$\tau_{J^c}[\chi_{\{h_{j,k,1}>t\}}]\leq\frac{\|h_{j,k,1}\|_p^p}{t^p}\leq\frac{\|h_{j,k}\|_p^p}{t^p}$$	
with $p=t^{\frac{2}{d-1}}/e>4$.	It then follows from hypercontractivity-Lemma \ref{basic} ii) and noting $\|h_{j,k}\|_2=1$ that
\begin{align}\label{e222w122222}
\frac{\|h_{j,k}\|_p^p}{t^p}\leq\frac{\big((p-1)^{\frac{d-1}{2}}\|h_{j,k}\|_2\big)^p}{t^p}\leq\Big(\frac{p^{\frac{d-1}{2}}}{t}\Big)^p\leq\exp\Big(-\frac{d-1}{2e}t^{\frac{2}{d-1}}\Big).
\end{align}
%	\begin{align}\label{e222w122222}
%\frac{\|h_{j,k}\|_p^p}{t^p}\leq\frac{\big((p-1)^{\frac{d-1}{2}}\|h_{j,k}\|_2\big)^p}{t^p}\leq\Big(\frac{(p-1)^{\frac{d-1}{2}}}{t}\Big)^p=\exp\Big(-\frac{d-1}{2e}t^{\frac{2}{d-1}}\Big).
%\end{align}
Therefore, by (\ref{e222w122222}) and Lemma \ref{l2}, one gets
		\begin{align*}
			II_{j,k}	\leq &\ \frac{2}{t_0} \, \int_{t_0}^\infty t^2\exp\Big(-\frac{d-1}{2e} \,
			t^{2/(d-1)}\Big) \, dt \, \tau_{J^c}[|g_{j,k,1}|^2]\\
			\leq	&\  \frac{2}{t_0} \,  \int_{t_0}^\infty t^2\exp\Big(-\frac{d-1}{2e}t^{2/(d-1)}\Big) \, dt \; \tau_{J^c}[|g_{j,k}|^2]\\
			\leq &\ 10e  \,  t_0^{2-\frac{2}{d-1}} \,  \exp\Big(-\frac{d-1}{2e}\cdot t_0^{2/(d-1)}\Big)  \, \tau_{J^c}[|g_{j,k}|^2].
		\end{align*}
		By these two bounds, %$\sum_{j\in J}(\langle h_j\sigma_{\{j\}},A\rangle)^2$
		$\sum_{k=1}^3\sum_{j\in J}	\Big(\tau_{J^c}[h_{j,k,1}g_{j,k,1}]\Big)^2$ is bounded above by
		\begin{align*}
			%	\sum\limits_{\substack{s\in\{0,1,2,3\} ^{J^c}\\|\mathrm{supp}s|=d-1}}|\widehat A_{\{j\}\cup s}|^2	\leq &\
			6t_0^2 \,  \mathcal V_J[A]
			+10e  \,  t_0^{2-\frac{2}{d-1}}\exp\Big(-\frac{d-1}{2e}\cdot t_0^{2/(d-1)}\Big)\sum_{k=1}^3\sum_{j\in J}\|g_{j,k}\|_2^2.
		\end{align*}
		%Consider the following restriction  of $A$:
		%$$A=\sum_{\substack{s\in \{0,1,2,3\}^{|J|}\\\mbox{supp}s\subseteq J}} A_{s,J} \sigma_s,\ \mbox{where}\ A_{s,J}=\sum_{\substack{t\in \{0,1,2,3\}^{|J^c|}\\\mbox{supp}t\subseteq J^c}}\widehat A_{t\cup s} \sigma_{t}.$$
		Note that
\[ \sum_{k=1}^3\sum_{j\in J}\|g_{j,k}\|_2^2=\sum_{\substack{v\in\{0,1,2,3\}^{n}\\|\mathrm{supp }\, v\cap J|=1\\ |\mathrm{supp }\, v|=d}}|\widehat A_{v}|^2\le \norm{A}{2}^2\le 1.\]
		%Therefore, one has
		%	$$\sum_{j\in J}\tau_{J^c}\Big[\Big|\tau_J[\sigma_{\{j\}}A]\Big|^2\Big]=\sum_{j\in J}\sum_{s\in \{0,1,2,3\}^{J^c}}|\widehat A_{s\cup \{j\}}|^2\leq\|A\|_2^2\leq1.$$
		Therefore, we find
		\begin{align}\label{e22222222}
			\sum_{k=1}^3\sum_{j\in J}	\Big(\tau_{J^c}[h_{j,k,1}g_{j,k,1}]\Big)^2	\leq6t_0^2  \, \mathcal V_J[A]+10e  \,  t_0^{2-\frac{2}{d-1}}\exp\Big(-\frac{d-1}{2e}\cdot t_0^{\frac{2}{d-1}}\Big).
		\end{align}
		Now we choose $t_0$ such that
		$$\exp\Big(-\frac{d-1}{2e}\cdot t_0^{\frac{2}{d-1}}\Big)=\mathcal V_J[A].$$
By a simple calculation,
		$$ t_0=\Big(\frac{2e}{d-1}\Big)^{\frac{d-1}{2}}\Big(\log\Big(\frac{1}{\mathcal V_J[A]}\Big)\Big)^{\frac{d-1}{2}}$$
		and by assumption (\ref{e22122222}), we have  $t_0>(8e)^{(d-1)/2}$, satisfying the  requirement for Lemma \ref{l2} used in estimating $II_{j,k}$. Finally, substituting $t_0$ into (\ref{e22222222}) and noting $t_0>1$, we get
		$$\sum_{k=1}^3\sum_{j\in J}	\Big(\tau_{J^c}[h_{j,k,1}g_{j,k,1}]\Big)^2\leq 40\Big(\frac{ 2e}{d-1}\Big)^{d-1}\mathcal V_J[A]\Big(\log\Big(\frac{1}{\mathcal V_J[A]}\Big)\Big)^{d-1},$$
		which is exactly (\ref{e222122222}). Clearly, the above argument also works for the other fifteen terms in (\ref{e2221212222}) with the same bound as $\sum_{k=1}^3\sum_{j\in J}	\Big(\tau_{J^c}[h_{j,k,1}g_{j,k,1}]\Big)^2$ stated in (\ref{e222122222}). Hence, we obtain (\ref{e2222222}). This completes the proof.
	\end{proof}

Now we can finish the proof of Lemma \ref{l8}. Throughout the proof, we denote
$$\mathcal V[A]=\sum_{j=1}^n\mathrm{Inf}_{j}^1[A]^2.$$
\begin{proof}[Proof of Lemma \ref{l8}.]
	Choose $J$ to be a $\frac{1}{d}$- random subset of $[n]$. Then for each $\nu\in\{0,1,2,3\}^n$ with $|\mathrm{supp} \, \nu|=d$,
\[ \mathrm{Pr}[ |J\cap {\mathrm{supp} \nu}|=1]=(1-\frac{1}{d})^{d-1}>1/e\ .\]
Denote $\Omega_\nu:=\{ J \ |\ |J\cap {\mathrm{supp}\ , \nu}|=1\}$. Then
	$$\mathbb E_J\Big[ \sum_{|\mathrm{supp} \, \nu|=d} |\widehat A_{\nu}|^2 \; 1_{\Omega_\nu} \Big] = \sum_{|\mathrm{supp} \, \nu|=d} |\widehat A_{\nu}|^2 \; {\rm Pr} [\Omega_\nu] \; \geq \; \frac{1}{e} \sum_{|\mathrm{supp} \, \nu|=d} |\widehat A_{\nu}|^2.$$
	\iffalse
	Clearly we have $$\sum_{|\mathrm{supp} \, \nu|=d} |\widehat A_{\nu}|^2 \; 1_{\Omega_\nu}  =
	\sum\limits_{\substack{s\in\{0,1,2,3\} ^{J^c}\\|\mathrm{supp} \, s|=d-1}}\sum\limits_{\substack{t\in \{0,1,2,3\}^{J}\\|\mathrm{supp} \, t|=1}}|\widehat A_{s\cup t}|^2 .$$
	 Hence,
	\begin{align*}%\label{e222222222}
		\sum_{|\mathrm{supp} \, \nu|=d}|\widehat A_{\nu}|^2\leq e\mathbb E_J\Big[\sum\limits_{\substack{s\in\{0,1,2,3\} ^{J^c}\\|\mathrm{supp} \, s|=d-1}}\sum\limits_{\substack{t\in \{0,1,2,3\}^{J}\\|\mathrm{supp} \, t|=1}}|\widehat A_{s\cup t}|^2\Big].
	\end{align*}
\fi

On the other hand, by the assumption (\ref{e22122222}), and noting that $\mathcal V_J[A]\leq \mathcal V[A]$, we see that  $\mathcal V_J[A]$ is in the segment $[0,\exp(-(d-1))]$. Consequently, by Lemma \ref{l3} we have
	\begin{align*}
		\sum_{|\mathrm{supp} \, \nu|=d}|\widehat A_{\nu}|^2	\leq &\ e\mathbb E_J\Big[ 10240\Big(\frac{2e}{d-1}\Big)^{d-1}\mathcal V_J[A] \; \Big(\log\Big(\frac{1}{\mathcal V_J[A]}\Big)\Big)^{d-1}\Big]\\
		\leq &\ \frac{ 10240e}{d}
		\Big(\frac{2e}{d-1}\Big)^{d-1}\mathcal V[A] \; \Big(\log\Big(\frac{d}{\mathcal V[A]}\Big)\Big)^{d-1},
	\end{align*}	
	where in the last inequality we used Jensen's inequality and the concavity of the function $x\mapsto x\log(1/x)^{d-1}$ in the segment $[0,\exp(-(d-1))]$.
	% ADbove since by (\ref{e22122222}) $\sum_{j=1}^n\mathrm{Inf}_{j}^1(A)^2\in [0,\exp(-(d-1))]$ for any $J\in[n]$,
	This finishes the proof.
\end{proof}

In order to prove Theorem \ref{thm:main3}, the following estimate of the Fourier coefficients at level one is also needed.
\begin{lem}\label{l5}
	Let $A\in M_2(\C)^{\otimes n}$. Then
$$\sum_{|\mathrm{supp} \, s|=1}\, |\widehat A_{s}|^2\leq 3 \, \mathcal V[A]$$
\end{lem}
\begin{proof}
Recall that
\[\sum_{|\mathrm{supp} \, s|=1}\, |\widehat A_{s}|^2=\sum_{j=1}^n\sum_{k=1}^3|\widehat{A}_{k(j)}|^2\]
where $k(j)=(0,\cdot\cdot\cdot,0,s_j,0,\cdot\cdot,0)$ with $s_j=k$. For each $k\in \{1,2,3\}$ and $j\in [n]$,
\[|\widehat{A}_{k(j)}|=|\tau(A\sigma_{k(j)})|=|\tau(d_j(A)\sigma_{k(j)})|\le \mathrm {Inf}^1_{j}(A)\ .\]
Summing over $j\in [n]$ and $k\in\{1,2,3\}$ yields the assertion.
\end{proof}

%Now we are ready to prove Theorem \ref{thm:main3}.
\begin{proof}[Proof of Theorem \ref{thm:main3}.]
%By (\ref{noise1222}), we see that	
%$$S_\delta[A]=\mathrm{Var}[T_{\sqrt{\delta}} A]=\sum_{s\neq0} \, \delta^{|\mathrm{supp} \, s|} \, |\widehat A_{s}|^2.$$
Observe that
$$S_\delta[A]=\sum_{s\neq0} \, \delta^{|\mathrm{supp} \, s|} \, |\widehat A_{s}|^2.$$
It suffices to consider the case that $\mathcal V[A]<1$ since $%\mathrm{Var}[T_{\sqrt{\delta}}
S_\delta[A]\leq \|A\|_2^2\leq1$.
Denote by $\lfloor x\rfloor$ the integer part of $x$. Define
	$$L=\lfloor \alpha\log(1/\mathcal V[A])\rfloor+1,\ \mbox{where}\ \alpha=\frac{1}{(1-\delta)+\log2e+3\log\log2e}.$$
Then we decompose
	$$S_\delta[A]=\sum_{0<|\mathrm{supp} \, s|\leq L}\delta^{|\mathrm{supp} \, s|} \, |\widehat A_{s}|^2+\sum_{|\mathrm{supp} \, s|> L}\delta^{|\mathrm{supp} \, s|} \, |\widehat A_{s}|^2,$$
and estimate each of the terms separately.	
%	For some positive integer $L$ that will be chosen later, we decompose

To bound the high degree term, since $\|A\|_2\leq1$ and $\delta\in(0,1)$, it follows from Parseval's identity that
	\begin{align}\label{e2222222222}
	\sum_{|\mathrm{supp} \, s|> L}\delta^{|\mathrm{supp} \, s|}|\widehat A_{s}|^2\leq\delta^{L+1}\sum_{|\mathrm{supp} \, s|> L}|\widehat A_{s}|^2\leq\delta^{L+1}.
	\end{align}
Since $\log x\geq x-1$ for $0<x<1$ and $0<\delta<1$, we conclude from (\ref{e2222222222}) that
%By definition of $L$ and $\alpha$,
\begin{align}\label{e122222222222}
	\sum_{|\mathrm{supp} \, s|> L}\delta^{|\mathrm{supp} \, s|}|\widehat A_{s}|^2\leq\delta^{L}\leq\exp((\delta-1)\alpha\log(1/\mathcal V[A]))=\mathcal V[A]^{\alpha(1-\delta)}.
\end{align}

Now we turn to estimate the low degree term.  By applying Lemma \ref{l5} to degree $d=1$, one has
\begin{align}\label{e2w222222222}
\sum_{|\mathrm{supp} \, s|=1}\delta^{|\mathrm{supp} \, s|} \, |\widehat A_{s}|^2\leq\sum_{|\mathrm{supp} \, s|=1}\, |\widehat A_{s}|^2\leq3 \, \mathcal V[A].
\end{align}
where we used Lemma \ref{l5}.
On the other hand, since $0<\alpha<1/2$, we can verify that for all $2\leq d\leq  L$,  $\mathcal V[A]$ satisfies the assumption  (\ref{e22122222}) stated in Lemma \ref{l8}. Hence, we apply Lemma \ref{l8}, using that degree $d$ with $2\leq d\leq  L$, to obtain that
$$\sum_{2\leq|\mathrm{supp} \, s|\leq L}\delta^{|\mathrm{supp} \, s|} \, |\widehat A_{s}|^2\leq\sum_{d=2}^L\frac{10240e}{d}\Big(\frac{2e}{d-1}\Big)^{d-1}\mathcal V[A]\Big(\log\Big(\frac{d}{\mathcal V[A]}\Big)\Big)^{d-1}.$$
Combining with (\ref{e2w222222222}), we  see that
 $\displaystyle\sum_{0<|\mathrm{supp} \, s |\leq L}\delta^{|\mathrm{supp} \, s|}|\widehat A_{s}|^2$ is bounded above by
	\begin{align*}%\label{e22222222222}
%	\sum_{d=1}^L\sum_{|\mathrm{supp}s|=d}|\widehat A_{s}|^2\leq
3\mathcal V[A]+\sum_{d=2}^L\frac{10240e}{d}\Big(\frac{2e}{d-1}\Big)^{d-1}\mathcal V[A]\Big(\log\Big(\frac{d}{\mathcal V[A]}\Big)\Big)^{d-1}:=3\mathcal V[A]+\sum_{d=2}^L R_d.
	\end{align*}
%Assume that $L\leq1/\log\mathcal V[A]$. Then
Observe that for any $d$ with $1<d<L$, $R_{d+1}\geq2R_d$. Indeed, %since 	$\alpha<1$
the ratio $R_{d+1}/R_d$ is
\begin{align*}
	&\
	\frac{d}{d+1}\Big(\frac{d-1}{d}\Big)^{d-1} \, \frac{2e}{d} \, \Big(\frac{\log((d+1)/\mathcal V[A])}{\log (d/\mathcal V[A])}\Big)^{d+1}\log\Big(\frac{d+1}{\mathcal V[A]}\Big)\\
	\geq&\  \frac{2}{d+1}\log\Big(\frac{d+1}{\mathcal V[A]}\Big)\geq\frac{2}{\log(1/\mathcal V[A])}\log\Big(\frac{d+1}{\mathcal V[A]}\Big)\geq2,
\end{align*}
where the second inequality holds since $d+1\leq L+1$ and $0<\alpha<1$.	
%\LG{more explanation needed for the last two steps above}
Hence, $2\sum_{d=2}^L R_d\leq \sum_{d=3}^L R_d+2R_L$, which implies $\sum_{d=2}^L R_d\leq2R_L$. Therefore,
%
%$$\frac{d}{d+1}\Big(\frac{d-1}{d}\Big)^{d-1}\frac{2e}{d}\Big(\frac{\log((d+1)/\mathcal V[A])}{\log (d/\mathcal V[A])}\Big)^{d+1}\log\Big(\frac{d+1}{\mathcal V[A]}\Big)\geq\frac{2}{\log(1/\mathcal V[A])}\log\Big(\frac{d+1}{\mathcal V[A]}\Big),$$
	\begin{align}\label{e22222222222}
	\sum_{0<|\mathrm{supp} \, s|\leq L}\delta^{|\mathrm{supp} \, s|}|\widehat A_{s}|^2\leq 3\mathcal V[A]+\frac{20480e}{L}\Big(\frac{2e}{L-1}\Big)^{L-1}\mathcal V[A]\Big(\log\Big(\frac{L}{\mathcal V[A]}\Big)\Big)^{L-1}.
\end{align}
%Then clearly $L\leq1/\mathcal V[A]$ since $\alpha<1$.
%Then since $\delta^L = \exp (L \log \delta)$, and s
To further bound the low degree term, the right-hand side of (\ref{e22222222222})  equals to
	\begin{align*}
		 &\
 3\mathcal V[A]+\frac{20480e}{L} \, \Big(\frac{2e}{\alpha}\Big)^{L-1} \, \Big(\frac{\alpha}{L}\Big)^{L-1}\Big(\frac{L}{L-1} \Big)^{L-1} \, \mathcal V[A] \, \Big(\log\Big(\frac{L}{\mathcal V[A]}\Big)\Big)^{L-1}\\
 \leq &\ 3\mathcal V[A]+\frac{20480e^2}{L}\Big(\frac{1}{\log(1/\mathcal V[A])}\Big)^{L-1}\Big(\frac{2e}{\alpha}\Big)^{L-1}\mathcal V[A] \, \Big(\log\Big(\frac{L}{\mathcal V[A]}\Big)\Big)^{L-1}\\
	= &\ 3\mathcal V[A]+\frac{20480e^2}{L}\Big(1+\frac{\log(L)}{\log(1/\mathcal V[A])}\Big)^{L-1}\mathcal V[A] \, \Big(\frac{2e}{\alpha}\Big)^{L-1}\\
	\leq &\ 3\mathcal V[A]+\frac{20480e^2}{L}\Big(1+\frac{\log(L)}{\log(1/\mathcal V[A])}\Big)^{\log(1/\mathcal V[A])}\mathcal V[A] \, \Big(\frac{2e}{\alpha}\Big)^{L-1}\\
	\leq &\ 3\mathcal V[A]+20480e^2 \, \mathcal V[A] \, \Big(\frac{2e}{\alpha}\Big)^{L-1},
\end{align*}
where the first inequality holds since $L\leq\alpha\log(1/\mathcal V[A])+1$ and %$L\leq\alpha\log(1/\mathcal V[A])$ \LG{should be $L\leq\alpha\log(1/\mathcal V[A])+1?$}and
the third inequality is due to $\alpha<1$. Since $\mathcal V[A]<1$, one has
\begin{align}\label{e2289}
	3 \, \mathcal V[A]\leq 3 \, \mathcal V[A]^{\alpha(1-\delta)}.
\end{align}	
%	Since $W\leq W^{\alpha\delta}$, it remains to show
In the following, we claim that
\begin{align}\label{e228}
	\mathcal V[A] \, \Big(\frac{2e}{\alpha}\Big)^{L-1}\leq \, \mathcal V[A]^{\alpha(1-\delta)}.
\end{align}	
Indeed, observe that
	\begin{align*}
		\mathcal V[A] \, \Big(\frac{2e}{\alpha}\Big)^{L-1}\leq	\, \mathcal V[A] \, \Big(\frac{2e}{\alpha}\Big)^{\alpha\log(1/	\mathcal V[A])}= \,	\mathcal V[A]^{1-\alpha\log(\frac{2e}{\alpha})}
	\end{align*}
%Hence, it is enough to show
%	$$1-\alpha\log\Big(\frac{2e}{\alpha}\Big)\geq\alpha\delta.$$
and
\begin{align*}
1-\alpha\log\Big(\frac{2e}{\alpha}\Big)
	= &\ \alpha\big[\frac{1}{\alpha}-\log2e-\log\frac{1}{\alpha}\big]\\
	=&\ \alpha(1-\delta)+\frac{3\log\log(2e)-\log1/\alpha}{(1-\delta)+\log(2e)+3\log\log(2e)}.
\end{align*}
%$$1-\alpha\log\Big(\frac{4e}{\alpha}\Big)=\alpha(1-\delta)+\frac{3\log\log(4e)-\log1/\alpha}{(1-\delta)+\log(4e)+3\log\log(4e)}.$$	
Note that $3\log\log(2e)-\log1/\alpha\geq0$, %which can be obtained by using the definition of $\alpha$,
$$3\log\log(2e)-\log1/\alpha\geq3\log\log(2e)-\log(1+\log2e+3\log\log2e)>0.$$
%so we conclude that $3\log\log(2e)-\log1/\alpha>0$.
This proves (\ref{e228}), since $\mathcal V[A]^t$ is a decreasing function of $t$.
Finally, by  (\ref{e122222222222}), (\ref{e2289})  and (\ref{e228}), we find
	$$S_\delta[A]\leq160000\, \mathcal V[A]^{\alpha(1-\delta)}.$$
	This finishes the proof.
\end{proof}

\section{Quantum FKN theorem}
 Recall that a Boolean function $f:\{-1,1\}^n\rightarrow\{-1,1\}$ is called a $k$-junta if it depends on at most $k$ of its input coordinates; i.e., $f(x)=g(x_{i_1},\cdot\cdot\cdot,x_{i_k})$ for some $g:\{-1,1\}^k\rightarrow\{-1,1\}$ and $i_1,\cdot\cdot\cdot,i_k\in[n]$. %If $k=1$, $g$ is called a dictator.
A Boolean function $f:\{-1,1\}^n\rightarrow\{-1,1\}$ is called a $(\varepsilon,k)$-junta if there exists a function $g:\{-1,1\}^n\rightarrow\{-1,1\}$ depending on at most $k$-coordinates such that
$$\|f-g\|_2^2\leq\varepsilon.$$
The celebrated FKN theorem, proved by Friedgut, Kalai, and Naor \cite{FKN},
 states that if a Boolean function $f$ has $W_{>1}[f]:=\sum_{|S|>1} |\widehat f(S)|^2=\varepsilon$, then $f$ is $O(\varepsilon)$-close to a 1-junta function. The FKN theorem is quite useful in the study of hardness of approximation and social choice
theory (c.f.\ \cite{K1,D1}).

The FKN theorem for quantum Boolean functions was first appeared in \cite[Theorem 9.7]{MO} with a minor fixable issue in the proof \footnotemark \footnotetext{The proof contains an incorrect statement ``$h^2\ge 1$'', which can be fixed following the classical argument \cite{O} .We are grateful to Ashley Montanaro for private communication on this}. Here we provide an alternative proof using the idea of the original argument from \cite{FKN}.
Recall that $A$ is a quantum Boolean function if A is a Hermitian unitary, i.e. $A=A^*$ and $A^2=1$.
In analogy with the classical setting, an operator $A\in M_2(\C)^{\otimes n}$ is called a $k$-{\em junta} if $|\mathrm{supp}(A)|\le k$, where the support of $A$ is defined by
$$\mathrm{supp}(A)=\bigcup_{s|\widehat A_s\neq0}\mathrm{supp}(s).$$
For convenience, we set
$$\ell=\sum_{\substack{s\in \{0,1,2,3\}^n\\|\mathrm{supp} \, s|=1}}\widehat A_s \, \sigma_s\ \; \; \;  \mbox{and}\ \; \; \; h=\sum_{\substack{s\in \{0,1,2,3\}^n\\|\mathrm{supp}  \, s|>1}}\widehat A_s \, \sigma_s.$$
\begin{lem}\label{FKN}
	Let $A\in M_2(\C)^{\otimes n}$ be a quantum Boolean function and $\varepsilon>0$. Assume that $\|h\|_2^2=\varepsilon$. Then there exist $j\in[n]$ and an absolute constant $K>0$ such that $A$ is $K\varepsilon$-close to
	a 1-junta operator $B_j$. That is
    $$\|A- B_j\|_2^2\leq K\varepsilon.$$
     %and
	%some $s\in \{0,1,2,3\}^n$ with $|\mathrm{supp} \, s|=1$
	%such that
	%, where $B_j$ is a  given by
	%$$B_j=\sum_{k=1,2,3}\sum_{\substack{s\in \{0,1,2,3\}^n\\|\mathrm{supp} \, s|=1,s_j=k}}\widehat A_s\sigma_s.$$
\end{lem}

%Recall that a function $f:\{-1,1\}\rightarrow\R$ is called a $k$-junta if  it depends on at most $k$ coordinates.
\begin{proof}
	According to \cite[Section 9.1]{MO}, we may assume that $\mathrm{tr}(A)=0$.	Then $A=\ell+h$. 	
Suppose that $\varepsilon\leq\frac{1}{10^6}$.
	By the  anti-commutation relation of Pauli
	matrices
	$$\sigma_j\sigma_k+\sigma_k\sigma_j=0,\ \qquad  j \neq k\in\{1,2,3\},$$
one has
	$\ell^2=(1-\varepsilon)\mathbb I+q$, where
	$$q=\sum_{\substack{|\mathrm{supp}s|=|\mathrm{supp} \, t|=1\\\mathrm{supp} \, s\cap\mathrm{supp} \, t=\emptyset}}\widehat A_s\widehat A_t \, \sigma_s\sigma_t.$$
Let $R=\ell^2-\mathbb I$. The proof will be finished if we can show that for any  $\alpha\in(0,1/2)$,
	\begin{equation}\label{1}
		\|R\|_2^2\leq \frac{\alpha^2}{1-\frac{54\sqrt{\varepsilon}}{\alpha}}.
	\end{equation}
	Indeed, choose $\alpha=108\sqrt{\varepsilon}$. Then (\ref{1}) implies $\|R\|_2^2\leq23328\varepsilon$. Recall that $k(j)=(0,\cdots,0,s_j,0\cdots, 0 )$ with $s_j=k$.
Since
 $R=-\varepsilon\mathbb{I}+q$, it follows from Parseval's identity that
	\begin{align*}
		23329\varepsilon
		\geq  \|q\|_2^2=&\ \sum_{\substack{s, t\in \{0,1,2,3\}^n\\|\mathrm{supp}s|=|\mathrm{supp} \, t|=1}}\widehat A_s^2\widehat A_t^2-\sum_{\substack{|\mathrm{supp}s|=|\mathrm{supp} \, t|=1\\|\mathrm{supp} \, s\cap\mathrm{supp} \, t|=1}}\widehat A_s^2\widehat A_t^2\\
=&\ (\sum_{\substack{s\in \{0,1,2,3\}^n\\|\mathrm{supp}s|=1}}\widehat A_s^2 )^2  -\sum_{j\in [n]}\sum_{k=1,2,3}\widehat A_{k(j)}^2\sum_{k'=1,2,3}\widehat A_{k'(j)}^2\\
		=&\ (1-\varepsilon)^2   -\sum_{j\in [n]}\sum_{k=1,2,3}\widehat A_{k(j)}^2\sum_{k'=1,2,3}\widehat A_{k'(j)}^2\\
		\geq&\  1-2\varepsilon-\max_{j\in [n]}\sum_{k=1,2,3}\widehat A_{k(j)}^2.
	\end{align*}
As a consequence, we have
$$\max_{j\in [n]}\sum_{k=1,2,3}\widehat A_{k(j)}^2\geq1-23331\varepsilon,$$
which implies that there is some $j\in[n]$ such that
\begin{equation}\label{182}
\sum_{k=1}^3\widehat A_{k(j)}^2\geq1-23331\varepsilon.
\end{equation}
If we define
\begin{equation}\label{183}
B_j=\sum_{k=1}^3\widehat A_{k(j)} \, \sigma_{k(j)}=E_{{j}^c}(A),
\end{equation}
then using Parseval's identity and (\ref{182}), one gets
	$$\|A-B_j\|_2^2=\varepsilon+\sum_{\substack{s\in \{0,1,2,3\}^n\\s_j=0}}\widehat A_t^2\leq23332\varepsilon,$$
which gives the desired assertion.

%	Hence, there is some $s\in \{0,1,2,3\}^n$ with $|\mathrm{supp}s|=1$ such that
%	$\widehat A_s^2\geq1-23331\varepsilon.$ As a consequence,

	%$$2915\varepsilon\geq\|q\|_2^2=\sum_{\substack{s, t\in \{0,1,2,3\}^n\\|\mathrm{supp}s|=|\mathrm{supp}t|=1}}\widehat A_s^2\widehat A_t^2-\sum_{\substack{s\in \{0,1,2,3\}^n\\|\mathrm{supp}s|=1}}\widehat A_s^4.$$
	It remains to show (\ref{1}). To see this, let $p=\tau\big(\chi_{\{|R|>\alpha\}}\big)$ where $\chi_{\{|R|>\alpha\}}$ is the spectrum projection of $|R|$. We claim that
	\begin{equation}\label{2}
		p\leq\frac{36\varepsilon}{\alpha^2}.
	\end{equation}
	Note that $R=(A-h)^2-\mathbb I=h^2-Ah-hA$. Then using a distribution inequality (see e.g. \cite[Lemma 16]{Su12})  and Chebyshev's inequality, we deduce that
	\begin{align*}
		p
		\leq&\  \tau\big(\chi_{\{|h^2|>\alpha/3\}}\big)+\tau\big(\chi_{\{|Ah|>\alpha/3\}}\big)+\tau\big(\chi_{\{|hA|>\alpha/3\}}\big)\\
		\leq&\ \frac{3}{\alpha}\|h\|_2^2+\frac{9}{\alpha^2}\|Ah\|_2^2+\frac{9}{\alpha^2}\|hA\|_2^2\leq\frac{36}{\alpha^2}\|h\|_2^2,
	\end{align*}
	as desired. Therefore, we get
	\begin{align*}
		\|R\|_2^2
		=&\  \tau\big(R^2\chi_{\{|R|\leq\alpha\}}\big)+\tau\big(R^2\chi_{\{|R|>\alpha\}}\big)\\
		\leq&\ \alpha^2(1-p)+\|R\|_4^2\cdot\sqrt{p}\\
		\leq&\ \alpha^2+\frac{54\sqrt{\varepsilon}}{\alpha}\|R\|_2^2,
	\end{align*}
	where in the first inequality we used Cauchy-Schwarz inequality and the second inequality follows from hypercontractivity Lemma \ref{basic} ii) and (\ref{2}).  This yields that
	$$	\|R\|_2^2\leq\frac{\alpha^2}{1-\frac{54\sqrt{\varepsilon}}{\alpha}},$$
	which finishes the argument of the case $\varepsilon\leq\frac{1}{10^6}$. The case $\varepsilon>\frac{1}{10^6}$ is trivial as
	$$\|A-B_j\|_2^2\leq2(\|A\|_2^2+\|B_j\|_2^2)\leq4\leq10^7\varepsilon.$$
	This completes the proof.
\end{proof}

We now present the quantum FKN theorem.
\begin{thm}\label{FKN1}
	Let $A\in M_2(\C)^{\otimes n}$ be a quantum Boolean function and $\varepsilon>0$. Assume that $\|h\|_2^2=\varepsilon$. Then there exist $j\in[n]$ and an absolute constant $K_1>0$ such that $A$ is $K_1\varepsilon$-close to
	a 1-junta quantum Boolean function $C_j$. That is,
	$$\|A- C_j\|_2^2\leq K_1\varepsilon.$$
	%and
	%some $s\in \{0,1,2,3\}^n$ with $|\mathrm{supp} \, s|=1$
	%such that
	%, where $B_j$ is a  given by
	%$$B_j=\sum_{k=1,2,3}\sum_{\substack{s\in \{0,1,2,3\}^n\\|\mathrm{supp} \, s|=1,s_j=k}}\widehat A_s\sigma_s.$$
\end{thm}	
\begin{proof}
	By Lemma \ref{FKN} and its proof, there are some $j\in[n]$, $K>0$ and a 1-junta operator $B_j$ given by (\ref{183}) such that $\|A- B_j\|_2^2\leq K\varepsilon$.
	Given the spectral decomposition $B_j=\sum_i\lambda_i|\psi_i\rangle\langle \psi_i|$,  we define the self-adjoint unitary
$$\mathrm{sgn}(B_j)=\sum_{i}\mathrm{sgn}(\lambda_i)|\psi_i\rangle\langle\psi_i|,$$ where the sign function is defined as $\mathrm{sgn}(x)=1$ if $x>0$ and $\mathrm{sgn}(x)=-1$ if $x\leq0$.	
Since $|\lambda +\mathrm{sgn}(\lambda)|\ge 1\ge |\lambda -\mathrm{sgn}(\lambda)|$, it follows that
	\begin{equation*}
		|\lambda-\mathrm{sgn}(\lambda)|^2\le |(\lambda-\mathrm{sgn}(\lambda))(\lambda+\mathrm{sgn}(\lambda))|^2=|\lambda^2-1|^2
	\end{equation*}
	and consequently,
	\begin{equation*}
		\|B_j-C_j\|_2^2=\frac{1}{2^n}\sum_{i}|\lambda_i-\mathrm{sgn}(\lambda_i)|^2\le \frac{1}{2^n}\sum_{i}|\lambda_i^2-1|^2= \|B_j^2-\1\|_2^2\,.
	\end{equation*}
	Therefore, we find
$$\|A-C_j\|^2_2\le 2\big(\|A-B_j\|^2_2+\|B_j-C_j\|^2_2\big)\le 2K\varepsilon+2\|B_j^2-\1\|^2_2.$$	
Since $A^2=\1$ and $\|B_j\|=\norm{E_{{j}^c}(A)}{}\leq\|A\|\leq1$, we have	
	\begin{align*}
	\|B_j^2-\1\|^2_2=
		\|B_j^2-A^2\|^2_2
		&\le 2\big(\|(B_j-A)B_j\|^2_2+\|A(B_j-A)\|^2_2\big)\leq4K\varepsilon,
	\end{align*}
	which gives that $\|A-C_j\|^2_2\le6K\varepsilon$. By letting $K_1=6K$, we finish the proof.	
\end{proof}

\section{Summary and Discussions}
We end this paper with some further discussions.
\subsection{Quantum $L^2$-KKL theorem}
Follow the work \cite{DJKR}, we use the method of random restriction to obtain a dimension free bound for quantum $L^1$-influence, which implies the quantum KKL-theorem for $L^1$ influence proved in \cite{RWZ}. Same as \cite{RWZ}, our method also works with $L^p$-influence with $1\le p<2$ but with a blow-up constant $C_p\to \infty$ as $p\to 2$ (see Remark \ref{quantumlp}). For higher order quantum $L^1$-influence, we obtain the KKL type theorem and Talagrand inequality with constant independent of order $k$. This extends the classical case work \cite{TP} to the quantum setting.
From this regard, the original conjecture of quantum KKL theorem (as well as high order case) for $L^2$-influence from \cite{MO} remains open.

\subsection{Quantum Bourgain's junta theorem}
Using  random restriction, we also prove a upper bound of the noise stability via square sum of quantum $L^1$-influence. We believe the method of random restriction adopted in this work will have potential applications in other problem in quantum Boolean analysis.
Nevertheless, there is a key difference in random restrictions between classical setting and quantum setting. That is, the restrictions of a Boolean function remains Boolean, but the restrictions of a quantum Boolean function (self-adjoint unitary) are usually not Boolean. This feature may leads to some limitation of applications of random restriction in the quantum setting. For instance, the Bourgain's junta theorem \cite{Bou02} that naturally reflects this phenomenon.

\medskip

%\noindent {\bf Acknowledgements} \

\end{document}